\documentclass[11pt,leqno]{amsart}
\usepackage{amsmath,amsthm,amsfonts,amssymb,oldgerm,mathrsfs}
\numberwithin{equation}{section}
\usepackage[frenchb,english]{babel}
\allowdisplaybreaks[1]
\usepackage{upref}
\usepackage{color}
\usepackage[citecolor=blue,colorlinks=true]{hyperref}

\newtheorem{theorem}{Theorem}[section]
\newtheorem{proposition}[theorem]{Proposition}

\newtheorem{lemma}[theorem]{Lemma}

\theoremstyle{definition}

\newtheorem{remark}[theorem]{Remark}
\newtheorem{example}[theorem]{Example}

\def\eps{\varepsilon }
\newcommand{\dd}{\textrm{d}}

\newcommand{\E}{\mathbb{E}}
\newcommand{\N}{\mathbb{N}}

\newcommand{\R}{\mathbb{R}}

\newcommand{\PP}{\mathbb{P}}

\newcommand\cA{{\mathcal A}}

\newcommand{\dif}{\mathrm{d}}

\newcommand{\mf}{\mathscr{F}}
\newcommand{\mr}{\mathbb{R}}
\newcommand{\prst}{\mathbb{P}}
\newcommand{\stred}{\mathbb{E}}

\newcommand{\mn}{\mathbb{N}}

\DeclareMathOperator{\diver}{div}

\newcommand{\semicol}{;}

\title[A regularity result for quasilinear SPDE'\MakeLowercase{s} of parabolic type]{A regularity result for quasilinear stochastic partial differential equations of parabolic type}

\author{Arnaud Debussche}
\address[A. Debussche]{IRMAR, ENS Rennes, CNRS, UEB, av. Robert Schuman, F-35170 Bruz, France}
\email{arnaud.debussche@ens-rennes.fr}
\author{Sylvain de Moor}
\address[S. De Moor]{IRMAR, ENS Rennes, CNRS, UEB, av. Robert Schuman, F-35170 Bruz, France}
\email{sylvain.demoor@ens-rennes.fr}
\author{Martina Hofmanov\'a}
\address[M. Hofmanov\'a]{Max Planck Institute for Mathematics in the Sciences, Inselstr. 22, 04103 Leipzig, Germany}
\email{mar.hofmanova@gmail.com}

\begin{document}
\begin{abstract}
We consider a non degenerate quasilinear parabolic stochastic partial differential equation with a uniformly elliptic diffusion matrix. It is driven by a nonlinear noise. We study regularity properties of its weak solution satisfying classical \textit{a priori} estimates. In particular, we determine conditions on coefficients and initial data under which the weak solution is H\"older continuous in time and possesses spatial regularity.
Our proof is based on an efficient method of increasing regularity: the solution is rewritten as the sum of two processes, one solves a linear parabolic SPDE with the same noise term as the original model problem whereas the other solves a linear parabolic PDE with random coefficients. This way, the required regularity can be achieved by repeatedly making use of known techniques for stochastic convolutions and deterministic PDEs.
\end{abstract}

\date{\today}
\maketitle


\section{Introduction}
In this paper, we are interested in the regularity of weak solutions of non degenerate
quasilinear parabolic stochastic partial differential equation driven by a multiplicative noise. 
Let $D\subset\mr^N$ be a bounded domain with smooth boundary, let $T>0$ and set 
$D_T= (0,T)\times D$, $S_T=(0,T]\times\partial D $. We study the following problem
\begin{equation}\label{equation}
\left\{
\begin{aligned}
\displaystyle \dd u&=\diver(B(u))\,\dif t+\mathrm{div}\left(A(u)\nabla u\right) \dd t+F(u)\,\dif t+H(u)\,\dd W,&\text{ in } &D_T,\\
u&=0, &\text{ on } &S_T,\\
u(0)&=u_0, &\text{ in } &D.
\end{aligned}
\right.
\end{equation}
where $W$ is a cylindrical Wiener process on some Hilbert space $K$ and $H$ is a mapping with values in the space of the $\gamma$-radonifying operators from $K$ to certain Sobolev spaces. The diffusion matrix $A$ is assumed to be smooth and uniformly elliptic and the initial condition $u_0$ is random in general. The precise description of the problem setting will be given in the next section.

It is a well known fact in the field of PDEs and SPDEs that many equations do not, in general, have 
classical or strong solutions and can be solved only in some weaker sense. Unlike deter\-ministic 
problems, in the case of stochastic equations we can only ask whether the solution is smooth in the 
space variable since the time regularity is limited by the regularity of the stochastic integral. Thus, the 
aim of the present work is to determine conditions on coefficients and initial data under which there 
exists a spatially smooth solution to \eqref{equation}. 

Such a regularity result is fundamental and 
interesting by itself. Equations of the form \eqref{equation} appear in many sciences. Regularity of 
solutions is an important property when one wants to study qualitative behaviour. It is also a 
preliminary step when studying numerical approximations and implies strong uniqueness for \eqref{equation}.
 Our original motivation is that such 
models arise as limits of random kinetic equations (see
\cite{DV}).

The issue of existence of a classical solution to deterministic parabolic problems is well understood, 
among the main references stands the extensive book \cite{lady} which is mainly concerned with the 
solvability of initial-boundary value problems and the Cauchy problem to the basic linear and 
quasilinear second order PDEs of parabolic type. Special attention is paid to the connection between 
the smoothness of solutions and the smoothness of known data entering into the problem (initial 
condition and coefficients), nevertheless, due to the technical complexity of the proofs a direct 
generalization to the stochastic case is not obvious.

In the case of linear parabolic problems, let us mention the classical Schauder theory (see e.g. 
\cite{lieberman}) that provides \textit{a priori} estimates relating the norms of solutions of initial-
boundary value problems, namely the parabolic H\"older norms, to the norms of the known quantities in 
the problems. These results are usually employed in order to deal with quasilinear equations: the 
application of the Schauder fixed point theorem leads to the existence of a smooth solution 
under suitable hypotheses on the coefficients. In our proof, we make use of the Schauder theory as 
well, yet in an entirely different approach.

Regularity of parabolic problems in the stochastic setting was also studied in several works. In the 
previous work of the third author \cite{hof}, semilinear parabolic SPDEs (i.e. the diffusion matrix $A$ being
independent of the solution) were studied and a regularity result established by using semigroup 
arguments (see also \cite{zhang}, \cite{zhang1}). In \cite{denis-matoussi}, \cite{dms}, a maximum principle is obtained for an SPDE similar to 
\eqref{equation} but with a more general diffusion $H$, it may depend on the gradient of $u$. In \cite{gess}, existence and uniqueness of strong solutions to SPDEs with drift given by the subdifferential of a quasi-convex function is proved. H\"{o}lder continuity of solutions to nonlinear parabolic systems under suitable structure 
conditions was proved in \cite{flandoli} by energy methods. Quasilinear stochastic porous media equations are studied in \cite{br}, \cite{br1} and specific techniques for these equations are used. In comparison to this work, the 
quasilinear case considered in the present paper is more delicate and different techniques need 
to be applied. 

The transposition of the deterministic method exposed in \cite{lady} seems to be quite difficult. Fortunately, we may use a trick to avoid this. We use a very simple idea: 
a weak solution to \eqref{equation} that satisfies \textit{a priori} estimates is decomposed into two parts 
$u=y+z$ where $z$ is a solution to a linear parabolic SPDE with the same noise term as 
\eqref{equation} and $y$ solves a linear parabolic PDE with random coefficients. As a consequence, 
the problem of regularity of $u$ is reduced to showing regularity of $z$ and regularity of $y$ which can 
be handled by known techniques for stochastic convolutions and deterministic PDEs. It is
rather surprising that this classical idea used to treat semilinear equations can be applied also for 
quasilinear problems. 

Let us explain this method more precisely. As the main difficulties come from the second order and 
stochastic terms, for simplicity of the introduction we assume $B=F=0$. Let $u$ be a weak solution to
\begin{equation}\label{equation1}
\left\{
\begin{split}
\displaystyle \dd u&=\mathrm{div}\left(A(u)\nabla u\right) \dd t+H(u)\,\dd W,\\
u(0)&=u_0,
\end{split}
\right. 
\end{equation}
and let $z$ be a solution to
\begin{equation*}
\left\{
\begin{split}
\displaystyle \dd z&=\Delta z \,\dd t+H(u)\,\dd W, \\
z(0)&=0.
\end{split} 
\right.
\end{equation*}
Then $z$ is given by the stochastic convolution with the semigroup generated by the Laplacian, denoted by $(S(t))_{t\geq0}$, i.e.
$$z(t)=\int_0^t S(t-s) H(u)\,\dif W(s)$$
and regularization properties are known.
Setting $y=u-z$ it follows immediately that $y$ solves
\begin{equation}\label{equation2}
\left\{
\begin{split}
\partial_t y&=\diver(A(u)\nabla y)+\diver((A(u)-\mathrm{I})\nabla z),\\
y(0)&=u_0,
\end{split}
\right.
\end{equation}
which is a (pathwise) deterministic linear parabolic PDE. According to \textit{a priori} estimates for \eqref{equation1}, it holds for all $p\in[2,\infty)$ and $u_0\in L^p(\Omega\times D)$:
$$u\in L^p(\Omega\semicol L^\infty(0,T\semicol L^p(D)))\cap L^p(\Omega\semicol L^2(0,T\semicol W^{1,2}(D))),$$
and making use of the factorization method it is possible to show that $z$ possesses enough 
regularity so that $\nabla z$ is a function with good integrability properties. Now, a classical result for 
deterministic linear parabolic PDEs with discontinuous coefficients (see \cite{lady}) yields H\"{o}lder 
continuity of $y$ (in time and space) and consequently also H\"{o}lder continuity of $u$ itself.
Having this in hand, the regularity of $z$ can be increased to a level where the Schauder theory for 
linear parabolic PDEs with H\"{o}lder continuous coefficients applies to \eqref{equation2} (see 
\cite{lieberman}) and higher regularity of $y$ is obtained. Repeating this approach then allows us to 
conclude that $u$ is $\lambda$-H\"{o}lder continuous in time for all $\lambda<1/2$ and possesses as 
much regularity in space as allowed by the regularity of the coefficients and the initial data. In this 
article, in order to avoid lengthy proofs and notations, we restrict to spatial regularity less than $5$. 
Our method extends to higher regularity, see Remark \ref{farther}.

The paper is organized as follows. In Section \ref{sec:notation}, we introduce the basic setting and 
state our regularity results, Theorem \ref{mainresult1}, Theorem \ref{mainresult2}. 
Section \ref{sec:preliminaries} gives preliminary results  concerning the stochastic convolution and the 
smoothness of the solutions of linear parabolic equations. These are adapted from \cite{lady} and 
\cite{lieberman}, we have to explicit the dependance of the constant with respect to some data and to
treat low time regularity. The remainder of the paper is devoted to the proof of Theorem 
\ref{mainresult1} and Theorem \ref{mainresult2} that is divided into several parts. In Section 
\ref{sec:first}, we establish our first regularity result, Theorem \ref{mainresult1}, that gives some 
H\"{o}lder continuity in time and space of a weak solution to \eqref{equation}. The regularity is then 
improved in the final Section \ref{sec:increase} and Theorem \ref{mainresult2} is proved.

\section{Notations, hypotheses and the main result}
\label{sec:notation}

\subsection{Notations}

In this paper, we adopt the following conventions. For $r\in[1,\infty]$, the Lebesgue spaces $L^r(D)$ are denoted by $L^r$ and the corresponding norm by $\|\cdot\|_r$. In order to measure higher regularity of functions we make use of the Bessel potential spaces $H^{a,r}(D)$, $a\in\mr$ and $r\in(1,\infty)$.

In order to motivate the use of these spaces let us recall their basic properties (for a thorough exposition we refer the reader to the books of Triebel \cite{triebel1}, \cite{triebel2}). In the case of $\mr^N$ the Bessel potential spaces are defined in terms of Fourier transform of tempered distributions: let $a\in\mr$, $r\in(1,\infty)$ then
$$H^{a,r}(\mr^N)=\big\{f\in\mathcal{S}'(\mr^N);\,\|f\|_{H^{a,r}}:=\big\|\mathcal{F}^{-1}(1+|\xi|^2)^{a/2}\mathcal{F}f\big\|_{L^r}<\infty\big\}$$
and they belong to the Triebel-Lizorkin scale $F^a_{r,s}(\mr^N)$ in the sense that $H^{a,r}(\mr^N)=F^{a}_{r,2}(\mr^N)$.
The Bessel potential spaces $H^{a,r}(\mr^N)$ behave well under the complex interpolation, i.e. for $a_0,a_1\in\mr$ and $r_0,r_1\in(1,\infty)$ it holds that
\begin{equation}\label{eq:interpol}
 [H^{a_0,r_0}(\mr^N),H^{a_1,r_1}(\mr^N)]_\theta=H^{a,r}(\mr^N),
\end{equation}
where $\theta\in(0,1)$ and $ a=(1-\theta)a_0+\theta a_1,\,\frac{1}{r}=\frac{1-\theta}{r_0}+\frac{\theta}{r_1}$,
which makes them more suitable for studying regularity for linear elliptic and parabolic problems. Indeed, under the assumption of bounded imaginary powers of a positive operator $\mathcal A$ on a Banach space $X$, the domains of fractional powers of $\mathcal A$ are given by the complex interpolation as well: let $0\leq\alpha<\beta<\infty$, $\theta\in(0,1)$ then
$$[D(\mathcal{A}^\alpha),D(\mathcal{A}^\beta)]_\theta=D(\mathcal{A}^{(1-\theta)\alpha+\theta\beta}).$$
Furthermore, the expression \eqref{eq:interpol} and the obvious identity $H^{m,r}(D)=W^{m,r}(D)$ for $m\in\mn_0$ and $r\geq 1$ suggest how the spaces $H^{a,r}(D)$ may be defined for a general domain $D$: if $a\geq0$ and $m\in\mn$ such that $a\leq m<a+1$ then we define
$$H^{a,r}(D):=[W^{m,r}(D),L^r(D)]_{(m-a)/m}.$$
If $D$ is sufficiently regular then $H^{a,r}(D)$ coincides with the space of restrictions to $D$ of functions in $H^{a,r}(\mr^N)$ and the Sobolev embedding theorem holds true. The spaces $H^{a,r}_0(D)$, $a\geq0$, $r\in(1,\infty)$, are then defined as the closure of $C_c^\infty(D)$ in $H^{a,r}(D)$. Note, that $H_0^{a,r}(D)= H^{a,r}(D)$ if $a\leq1/r$ and $H_0^{a,r}(D)$ is strictly contained in $H^{a,r}$ if $a>1/r$. Besides, an interpolation result similar to \eqref{eq:interpol} holds for these spaces as well
$$[H^{a_0,r_0}_0(D),H^{a_1,r_1}_0(D)]_\theta=H^{a,r}_0(D).$$
For notational simplicity, we denote the norm of $H^{a,r}$ by $\|\cdot\|_{a,r}$. 

\begin{remark}
The spaces $H^{a,r}(D)$ are generally different from the Sobolev-Slobodeckij spaces $W^{a,r}(D)$ which belong to the Besov scale $B^a_{r,s}(D)$ in the sense that $W^{a,r}(D)=B^a_{r,r}(D)$ if $a>0,\,a\notin\mn$.
Nevertheless, we have the following two relations which link the two scales of function spaces together
$$W^{a,r}(D)=H^{a,r}(D)\quad\text{ if }\quad a\in\mn_0,\,r\in[1,\infty)\quad \text{ or }\quad a\geq0,\,r=2,$$
and
$$H^{a+\varepsilon,r}(D)\hookrightarrow W^{a,r}(D)\hookrightarrow H^{a-\varepsilon,r}(D) \qquad a\in\mr,\,r\in(1,\infty),\,\varepsilon>0.$$
\end{remark}

Below we use the Laplace operator with Dirichlet boundary conditions, denoted by $\Delta_D$. Considered as an operator on $L^{r}$, its domain is $H^{2,r}_0$ and it is the infinitesimal generator of an analytic semigroup denoted by $S=(S(t))_{t\geq 0}$. Moreover, it follows from the above considerations that the domains of its fractional powers coincide with the Bessel potential spaces, that is $D((-\Delta_D)^\alpha)=H^{2\alpha,r}_0$, $\alpha\geq0$. Therefore, one can build a fractional power scale (or a Sobolev tower, see \cite{amann1}, \cite{engel}) generated by $(L^r,-\Delta_D)$ to get
\begin{equation}\label{eq:powerscale}
\big[\big(H_0^{2\alpha,r},-\Delta_{D,2\alpha,r}\big);\,\alpha\geq0\big],
\end{equation}
where $-\Delta_{D,2\alpha,r}$ is the $H_0^{2\alpha,r}$-realization of $-\Delta_D$.
Having this in hand, an important result \cite[Theorem V.2.1.3]{amann1} describes the behavior of the semigroup $S$ in this scale. More precisely, the operator $\Delta_{D,2\alpha,r}$ generates an analytic semigroup $S_{2\alpha,r}$ on $H_0^{2\alpha,r}$ which is naturally obtained from $S$ by restriction, i.e. $S_{2\alpha,r}(t)$ is the $H_0^{2\alpha,r}$-realization of $S(t)$, $t\geq0$, and we have the following regularization property: for any $\delta>0$ and $t>0$, $S_{2\alpha,r}(t)$ maps $H_0^{2\alpha,r}$ into $H_0^{2\alpha+\delta,r}$ with
\begin{equation}\label{regul}
\big\|S_{2\alpha,r}(t)\big\|_{\mathcal{L}(H_0^{2\alpha,r},H_0^{2\alpha+\delta,r})}\leq \frac{C}{t^{\delta/2}}.
\end{equation}
For notational simplicity of the sequel we do not directly specify the spaces where the operators $\Delta_D$ and $S(t)$, $t\geq0$, are acting since this is always clear from the context.

\medskip
Another important scale of function spaces which is used throughout the paper are the H\"{o}lder spaces. In particular, if $X$ and $Y$ are two Banach spaces and $\alpha\in(0,1)$, $C^{\alpha}(X;Y)$ denotes the space of bounded H\"{o}lder continuous functions with values in $Y$ equipped with the norm
$$\|f\|_{C^{\alpha}(X;Y)}\ =\ \sup\limits_{x\in X}\|f(x)\|_Y\ + \sup\limits_{x,x'\in X, x\neq x'}\frac{\|f(x)-f(x')\|_Y}{\|x-x'\|_X^{\alpha}}.$$ 
In the sequel, we consider the spaces $C^{\alpha}(\overline D)=C^{\alpha}(\overline D;\R)$, $C^{\alpha}([0,T];Y)$ where $Y=H^{a,r}$ or $Y=C^{\beta}(\overline D)$ and $C^{\alpha}([0,T]\times\overline D)=C^{\alpha}([0,T]\times \overline D;\R)$... Besides, we employ H\"{o}lder spaces with different regularity in time and space, i.e. $C^{\alpha,\beta}([0,T]\times\overline D)$. For $\alpha \in (0,1/2)$, $\beta\in (0,1)$, they are equipped with the norm
$$\|f\|_{C^{\alpha,\beta}}=\sup_{(t,x)}|f(t,x)|+\sup_{(t,x)\neq(s,y)}\frac{|f(t,x)-f(s,y)|}{\max\{|t-s|^{\alpha},|x-y|^\beta\}}.$$
For larger indices $\alpha+l/2$, $\beta+k$ with $\alpha \in (0,1/2)$, $\beta\in (0,1)$ and $l,k$  non 
negative integers, the norm is defined by
$$
\|f\|_{C^{\alpha+l/2,\beta+k}}=\sum_{2r\le l, |\gamma|\le k, 2r+|\gamma|\le \mathrm{max}\{l,k\}} \|\partial_t^r\partial^\gamma f\|_{C^{\alpha,\beta}}.
$$
We have denoted by $\partial_t$ the partial derivative with respect to the time variable $t$ and, 
for a multi-index $\gamma=(\gamma_1,\dots,\gamma_N)$, $\partial^\gamma =
\partial_1^{\gamma_1}\dots \partial_N^{\gamma_N}$ where $\partial_i$ is the partial derivative
with respect to $x_i$. For $\alpha=\beta/2$, $k=l$, these are the classical H\"older spaces used to measure the
regularity of solutions of parabolic problems. Here, we need these slightly more general 
spaces.  In this work, we always use $\alpha=\beta/2$. This is a very natural choice since in this case 
$C^{\alpha,\beta}$ is precisely the H\"older space of order $\beta$ with respect to the parabolic distance: $d((t,x),(s,y))=\max\{|t-s|^{1/2},|x-y|\}.$
\smallskip

Note that $C^{\beta/2,\beta+1}([0,T]\times\overline D)=C^{(\beta+1)/2,\beta+1}([0,T]\times\overline D)$.
Indeed both spaces consist of bounded functions such that 
$$
\|f\|_{C^{\beta/2,\beta}}+\sum_{ |\gamma|=1} \|\partial^\gamma f\|_{C^{\beta/2,\beta}}
$$
is finite.
\smallskip

Given a domain $D$ with a $C^{\beta+k}$ boundary, using local coordinates, we define classically
the H\" older spaces on the boundary $C^{\alpha,\beta}([0,T]\times\partial D)$ (see \cite{lady}, section 3, chapter II. In particular (3.19) and the paragraph above). 

Clearly if $v$ is a function in $C^{\alpha+l/2,\beta+k}([0,T]\times\overline D)$, its restriction to 
$[0,T]\times \partial D$ is in $C^{\alpha+l/2,\beta+k}([0,T]\times\partial D)$. But it can be much smoother,
for instance if $v$ is constant on the boundary.

Note that it holds $C^{\alpha}([0,T];C^{\beta}(\overline D))\subsetneqq C^{\alpha,\beta}([0,T]\times\overline D)$ and therefore we have to distinguish these two spaces. An example is given in \cite{holder}: $\overline{D }=[0,T]$ and $u(t,x)=(x+t)^\alpha$, $\alpha\in(0,1)$.
 
For $k\in\mn_0$ we denote by $C^k_b=C^k_b(\mr)$ the space of continuous functions that have continuous bounded derivatives up to order $k$. Note that for $k=0$ it is the space of continuous functions, not necessarily bounded.

In the whole article, $C,\, C_i,\,  K,\, K_i, \, \kappa, \, \dots$ denote constants. When they depend on some parameters of the problem, this is explicitly stated.

\subsection{Hypotheses}

Let us now introduce the precise setting of \eqref{equation}. We work on a finite-time interval $[0,T],\,T>0,$ and on a bounded domain $D$ in $\R^N$ with smooth boundary. We denote by $D_T$ the cylinder $(0,T)\times D$ and by $S_T$ the lateral surface of $D_T$, that is $S_T=(0,T]\times \partial D$. Concerning the coefficients $A,\,B,\,F,\,H$, we only state here the basic assumptions that guarantee the existence of a weak solution and are valid throughout the paper. Further regularity hypotheses are necessary in order to obtain better regularity of the weak solution and will be specified later.
We assume that the flux function
$$B=(B_1,\dots,B_N):\mr\longrightarrow \mr^N$$
is continuous with linear growth.
The diffusion matrix
$$A=(A_{ij})_{i,j=1}^N:\mr\longrightarrow\mr^{N\times N}$$
is supposed to be continuous, symmetric, positive definite and bounded. In particular, there exist constants $\nu$, $\mu>0$ such that for all $u\in\R$ and $\xi\in\R^N$,
\begin{equation}\label{elliptic}
\nu |\xi|^2\leq A(u)\xi\cdot\xi\leq \mu |\xi|^2.
\end{equation}
The drift coefficient $F:\mr\rightarrow\mr$ is continuous with linear growth.

Regarding the stochastic term, let $(\Omega,\mf,(\mf_t)_{t\geq0},\prst)$ be a stochastic basis with a complete, right-continuous filtration.
The driving process $W$ is a cylindrical Wiener process: $W(t)=\sum_{k\geq1}\beta_k(t) e_k$ with $(\beta_k)_{k\geq1}$ being mutually independent real-valued standard Wiener processes relative to $(\mf_t)_{t\geq0}$ and $(e_k)_{k\geq1}$ a complete orthonormal system in a separable Hilbert space $K$. For each $u\in L^2(D)$ we consider a mapping $\,H(u):K\rightarrow L^2(D)$ defined by $H(u)\,e_k=H_k(\cdot,u(\cdot))$. In particular, we suppose that $H_k\in C( D\times\mr)$ and the following linear growth condition holds true
\begin{equation}\label{growthhk}
\sum_{k\geq1}|H_k(x,\xi)|^2\leq C\big(1+|\xi|^2\big),\qquad\forall x\in D,\,\xi\in\mr.
\end{equation}
This assumption implies in particular that $H$ maps $L^2(D)$ to $L_2(K;L^2(D))$
where $L_2(K;L^2(D))$ denotes the collection of Hilbert-Schmidt operators from $K$ to $L^2(D)$. Thus, given a predictable process $u$ that belongs to $L^2(\Omega;L^2(0,T;L^2(D)))$, the stochastic integral $t\mapsto\int_0^t H(u)\dif W$ is a well defined process taking values in $L^2(D)$ (see \cite[Chapter 4]{daprato} for a thorough exposition).

Later on we are going to estimate the weak solution of $(\ref{equation})$ in certain Bessel potential spaces $H^{a,r}_0$ with $a\geq 0$ and $r\in[2,\infty)$ and therefore we need to ensure the existence of the stochastic integral in \eqref{equation} as an $H^{a,r}_0$-valued process. We recall that the Bessel potential spaces $H^{a,r}_0$ with $a\geq 0$ and $r\in[2,\infty)$ belong to the class of $2$-smooth Banach spaces since they are isomorphic to $L^r(0,1)$ according to \cite[Theorem 4.9.3]{triebel1} and hence they are well suited for the stochastic It\^o integration (see \cite{gamma1}, \cite{gamma2} for the precise construction of the stochastic integral). So, let us denote by $\gamma(K,X)$ the space of the $\gamma$-radonifying operators from $K$ to a $2$-smooth Banach space $X$. We recall that $\Psi\in\gamma(K,X)$ if the series 
$$\sum_{k\geq 0}\gamma_k\Psi(e_k)$$
converges in $L^2(\widetilde{\Omega},X)$, for any sequence $(\gamma_k)_{k\geq 0}$ of independent Gaussian real-valued random variables on a probability space $(\widetilde{\Omega},\widetilde{\mathcal{F}},\widetilde{\PP})$ and any orthonormal basis $(e_k)_{k\geq 0}$ of $K$. Then, the space $\gamma(K,X)$ is endowed with the norm 
$$\|\Psi\|_{\gamma(K,X)}:=\Bigg{(}\widetilde{\E}\Bigg{|}\sum_{k\geq 0}\gamma_k\Psi(e_k)\Bigg{|}_X^2\Bigg{)}^{1/2}$$ 
(which does not depend on $(\gamma_k)_{k\geq 0}$, nor on $(e_k)_{k\geq 0}$) and is a Banach space.
Now, if $a\geq 0$ and $r\in[2,\infty)$ we denote by \eqref{compgamma} the following hypothesis
\begin{equation}\label{compgamma}\tag{$\mathrm{H}_{a,r}$}
\begin{split}
\|H(u)\|_{\gamma(K,H^{a,r}_0)}&\leq \begin{cases}
							C\big(1+\|u\|_{H^{a,r}_0}\big),&\quad a\in[0,1],\\
							C\big(1+\|u\|_{H^{a,r}_0}+\|u\|_{H^{1,ar}_0}^a\big),&\quad a>1,
							\end{cases}\\
\end{split}
\end{equation}
i.e. $H$ maps $H^{a,r}_0$ to $\gamma(K,H^{a,r}_0)$ provided $a\in[0,1]$ and it maps $H^{a,r}_0\cap H^{1,ar}_0$ to $\gamma(K,H^{a,r}_0)$ provided $a>1.$ The precise values of parameters $a$ and $r$ will be given later in each of our regularity results.
\begin{remark}\label{H0r}
We point out that, thanks to the linear growth hypothesis \eqref{growthhk} on the functions $(H_k)_{k\geq 1}$, one can easily verify that, for all $r\in[2,\infty)$, the bound (\hyperref[compgamma]{$\text{H}_{0,r}$}) holds true. 
\end{remark}

In order to clarify the assumption \eqref{compgamma}, let us present the main examples we have in mind.

\begin{example}\label{ex1}
Let $W$ be a $d$-dimensional $(\mf_t)$-Wiener process, that is $W(t)=\sum_{k=1}^d W_k(t)\, e_k$, where $W_k,\,k=1,\dots,d,$ are independent standard $(\mf_t)$-Wiener processes and $(e_k)_{k=1}^d$ is an orthonormal basis of $K=\mr^d$. Then hypothesis \eqref{compgamma} is satisfied for $a\geq 0,\, r\in[2,\infty)$ provided the functions $H_1,\dots,H_d$ are sufficiently smooth and respect the boundary conditions in the following sense: if $a>\frac{1}{r}$, then
$$\nabla^l_x H_k(x,0)=0,\qquad x\in\partial D,\quad\forall k=1,\dots,d,\quad\forall l\in\mn_0,\,l< a-\frac{1}{r},$$
(for more details we refer the reader to \cite{runst}).
Note that in this example it is necessary to restrict ourselves to the subspace $H_0^{a,r}\cap H_0^{1,ar}$ of $H_0^{a,r}$ so that the corresponding Nemytskij operators $u\mapsto H_k(\cdot,u(\cdot))$ take values in $H_0^{a,r}$. In fact, if $1+1/r\leq a\leq N/r$, $r\in(1,\infty),$ then only linear operators map $H_0^{a,r}$ to itself (see \cite{runst}).
\end{example}

\begin{example}\label{ex2}
In the case of linear operator $H$ we are able to deal with an infinite dimensional noise. Namely, let $W$ be a $(\mf_t)$-cylindrical Wiener process on $K=L^2(D)$, that is $W(t)=\sum_{k\geq 1} W_k(t)\, e_k$, where $W_k,\,k\geq 1,$ are independent standard $(\mf_t)$-Wiener processes and $(e_k)_{k\geq 1}$ an orthonormal basis of $K$. We assume that $H$ is linear of the form  $H(u) e_k:=u\,Qe_k,\,k\geq 1$, where $Q$ denotes a linear operator from $K$ to $K$. Then, one can verify that the hypothesis \eqref{compgamma} is satisfied for $a\geq 0,\, r\in[2,\infty)$ provided we assume the following regularity property: $\sum_{k\geq 1}\|Qe_k\|^2_{a,\infty} <\infty$. We point out that, in this example, $H$ maps $H^{a,r}_0$ to $\gamma(K,H^{a,r}_0)$ for any $a\geq 0$ and $r\in[2,\infty)$.
\end{example}

As we are interested in proving the regularity up to the boundary for weak solutions of \eqref{equation}, it is necessary to impose certain compatibility conditions upon the initial data and the null Dirichlet boundary condition. To be more precise, since $u_0$ can be random in general, let us assume that $u_0:\Omega\rightarrow C(\overline{ D})$ is measurable and $u_0=0$ on $\partial D$. 
Further integrability and  regularity  assumptions on $u_0$ will be specified later.

Note that other boundary conditions could be studied with similar arguments, see Remark \ref{farther}
below.

\subsection{Existence of weak solutions}

Let us only give a short comment here as existence is not our main concern and we will only make use of \textit{a priori} estimates for parabolic equations of the form \eqref{equation}. In the recent work \cite{dehovo}, the authors gave a well-posedness result for degenerate parabolic SPDEs (with periodic boundary conditions) of the form
\begin{equation*}
\left\{
\begin{split}
\dif u&=\diver(B(u))\,\dif t+\diver(A(u)\nabla u)\,\dif t+H(u)\,\dif W,\\
u(0)&=u_0,
\end{split}
\right.
\end{equation*}
where the diffusion matrix was supposed to be positive semidefinite. One can easily verify that the Dirichlet boundary conditions and the drift term $F(u)$ in \eqref{equation} do not cause any additional difficulties in the existence part of the proofs and therefore the corresponding results in \cite{dehovo}, namely Section 4 (with the exception of Subsection 4.3) and Proposition 5.1, are still valid in the case of \eqref{equation}. In particular, we have the following.

\begin{theorem}\label{thm:weaksol}
There exists $\big((\tilde{\Omega},\tilde{\mf},(\tilde{\mf}_t),\tilde{\prst}),\tilde{W},\tilde{u}\big)$ which is a weak martingale solution to \eqref{equation} and, for all $p\in[2,\infty)$ and $u_0\in L^p(\tilde\Omega;L^p)$
$$\tilde u\in L^p(\tilde \Omega\semicol C([0,T]\semicol L^2))\cap L^p(\tilde \Omega\semicol L^\infty(0,T\semicol L^p))\cap L^p(\tilde\Omega\semicol L^2(0,T\semicol W^{1,2})).$$
\end{theorem}

By a weak martingale solution we mean that it is weak in the PDE sense, that is, the equation \eqref{equation} is satisfied in $\mathcal{D}'([0,T)\times D)$. It is also weak in probabilistic sense, that is, the probability space and the driving Wiener process are considered as unknown (see \cite[Chapter 8]{daprato} for precise definition).

In the sequel, we assume the existence of a weak solution (in the PDE sense) on the original probability space $(\Omega,\mf,\prst)$ and show that it possesses regularity that depends on the regularity of coefficients and initial data. We point out that this assumption is taken without loss of generality since pathwise uniqueness can be proved once we have sufficient regularity in hand and hence existence of a pathwise solution can be then obtained by the method of Gy\"ongy and Krylov \cite{krylov} (see also \cite[Subsection 4.3]{dehovo}).


\subsection{The main result}

To conclude this section let us state our main results to be proved precisely.\

\begin{theorem}\label{mainresult1}
Let $u$ be a weak solution to \eqref{equation} such that, for all $p\in[2,\infty)$,
$$u\in L^2(\Omega\semicol C([0,T]\semicol L^2))\cap L^p(\Omega\semicol L^\infty(0,T\semicol L^p))\cap L^2(\Omega\semicol L^2(0,T\semicol W^{1,2}_0)).$$
Assume that
\begin{enumerate}
 \item $u_0\in L^m(\Omega;C^{\iota}(\overline{D}))$ for some $\iota>0$ and all $m\in[2,\infty)$, and $u_0=0$ on $\partial D$ a.s.
 
 \item {\em(\hyperref[compgamma]{$\mathrm{H}_{1,2}$})} is fulfilled.
\end{enumerate}
Then there exists $\eta>0$ such that, for all $m\in[2,\infty)$, the weak solution $u$ belongs to $L^m(\Omega;C^{\eta/2,\eta}(\overline{D_T}))$.
\end{theorem}

Next result gives higher regularity. As in the deterministic case, some compatibility conditions are required on the boundary $\{0\}\times \partial D$. To state these, we introduce the notation: 
$u_0^{(1)}=\diver\left(A(u_0)\nabla u_0\right) + \diver(B(u_0))+F(u_0)$ and $\cA_0=\sum_{ij} A_{ij}(0)\partial_{ij}$.
\begin{theorem}\label{mainresult2}
Let $k=1,\, 2,\, 3$ or $4$. Let $u$ be a weak solution to \eqref{equation} such that, for all $p\in[2,\infty)$,
$$u\in L^2(\Omega\semicol C([0,T]\semicol L^2))\cap L^p(\Omega\semicol L^\infty(0,T\semicol L^p))\cap L^2(\Omega\semicol L^2(0,T\semicol W^{1,2}_0)).$$ Assume that
\begin{enumerate}
 \item $u_0\in L^m(\Omega;C^{k+\iota}(\overline{D}))$ for some $\iota>0$ and all $m\in[2,\infty)$,
 \item $u_0=0$ on $\partial D$ a.s., and $u_0^{(1)}=0$ on $\partial D$ a.s. if $k=2,\,3,\, 4$, and 
 $ 2 (A'(0)\nabla u_0)\cdot \nabla u_0^{(1)} + B'(0)\cdot \nabla u_0^{(1)} + \cA_0 u_0^{(1)}=0$ on $\partial D$ if $k=4$ .
 \item $A,\,B\in C^k_b$ and $F\in C^{k-1}_b$,
 \item \eqref{compgamma} is fulfilled for all $a<k+1$ and $r\in[2,\infty)$.
\end{enumerate}
Then for all $\lambda\in(0,1/2)$ and all $m\in[2,\infty)$, the weak solution $u$ belongs to $L^m(\Omega\semicol C^{\lambda,k+\iota}(\overline{D_T}))$.
\end{theorem}

\medskip

\begin{remark}\label{farther}
We could investigate higher regularity. Indeed, the tools required to prove Theorem \ref{mainresult2}
extend without difficulty. In fact, the proofs require different arguments for $k=1,2,3$ and $4$. But 
for $k\ge 5$, the proof is exactly the same as for $k=4$. This requires just a generalization of
Theorem \ref{t3.4} below. Note that, as in the deterministic case, this requires stronger assumptions
and compatibility conditions. The case of Neumann boundary conditions can be treated similarly. 

Our result clearly extends to the case of coefficients $A,\, B, \, F$ which may also depend on $x$ under
suitable regularity assumptions.

The case of periodic boundary conditions is 
much easier. Indeed, no compatibility conditions
are required. Moreover, the proof of regularity higher than $3$ is straightforward by differentiation
of the equation.
\end{remark}
\begin{remark}
Note, that the condition $m\in[2,\infty)$ in the assumption (i) of Theorem \ref{mainresult1} and Theorem \ref{mainresult2} can be weakened to $m\in[2,m_0]$ for some sufficiently large $m_0$. This would weaken the  results accordingly. However, in the proof, we use big vallues of $m$ and the precise
value of $m_0$ is not of great interest. Therefore, we prefered to state our result in the form above.

\end{remark}

\section{Preliminaries}
\label{sec:preliminaries}

\subsection{Regularity of the stochastic convolution}
\label{sec:convolution}

Our proof of Theorem \ref{mainresult1} and Theorem \ref{mainresult2} is based on a regularity result that concerns mild solutions to linear SPDEs of the form
\begin{equation}\label{aux123}
\left\{
\begin{aligned}
\displaystyle \dd Z&=\Delta_D Z\,\dif t+\Psi(t)\,\dd W_t,\\
Z(0)&=0,
\end{aligned}
\right.
\end{equation}
where $\Delta_D$ is the Laplacian on $D$ with null Dirichlet boundary conditions acting on various Bessel potential spaces.

The solution to \eqref{aux123} is given by the stochastic convolution, that is
$$Z(t)=\int_0^tS(t-s)\Psi(s)\,\dd W_s,\quad t\in[0,T].$$
In order to describe the connection between its regularity and the regu\-larity of $\Psi$, we recall the following proposition.

%

\begin{proposition}\label{RegulStochastic}
Let $a\geq 0$ and $r\in[2,\infty)$ and let $\Psi$ be a progressively measurable process in $L^p(\Omega;L^p(0,T;\gamma(K, H^{a,r}_0)))$.
\begin{enumerate}
\item Let $p\in(2,\infty)$ and $\delta\in[0,1-2/p)$. Then, for any $\lambda\in[0,1/2-1/p-\delta/2)$, $Z\in L^p(\Omega;C^{\lambda}(0,T;H^{a+\delta,r}_0))$ and
$$\E\|Z\|^p_{C^{\lambda}(0,T;H^{a+\delta,r}_0)}\leq C\,\E\|\Psi\|^p_{L^p(0,T;\gamma(K,H^{a,r}_0))}.$$
\item Let $p\in[2,\infty)$ and $\delta\in(0,1)$. Then $Z\in L^p(\Omega;L^p(0,T;H^{a+\delta,r}_0))$ and
$$\E\|Z\|^p_{L^p(0,T;H^{a+\delta,r}_0)}\leq C\,\E\|\Psi\|^p_{L^p(0,T;\gamma(K,H^{a,r}_0))}.$$
\end{enumerate}

\begin{proof}
Having established the behavior of the Dirichlet Laplacian and the corresponding semigroup along the fractional power scale \eqref{eq:powerscale}, the proof of (i) is an application of the factorization method and can be found in \cite[Corollary 3.5]{gamma1} whereas the point (ii) follows from the Burkholder-Davis-Gundy inequality and regularization properties \eqref{regul} of the semigroup.
\end{proof}

\end{proposition}

\subsection{Regularity results for deterministic parabolic PDE's}\label{s3.2}

We now state classical regularity results from \cite{lady} and \cite{lieberman}. Since the notations are different and difficult to find in the books, we restate the results with ours and in the situation needed in the sequel. Moreover, the dependence on the coefficients and initial data is not always explicit in 
these books. We thus precise the bounds, this requires new proofs. Also, we give new results where
the regularity is not measured in the classical parabolic scaling. 

We first consider a linear parabolic PDE of the form
\begin{equation}\label{eq}
\left\{
\begin{aligned}
\partial_t v&=\diver(a(t,x)\nabla v)+\diver g(t,x)+f(t,x),&\text{ in }&D_T,\\
v&=0,&\text{ on }&S_T,\\
v(0)&=v_0,&\text{ in }&D,
\end{aligned}
\right.
\end{equation}
and assume that there exist $\nu,\mu>0$ such that $\nu|\xi|^2\leq a(t,x)\xi\cdot\xi\leq\mu|\xi|^2$, for all $(t,x)\in D_T$, $\xi\in\mr^N$.

This equation is precisely (1.9) with $\mathscr{L}$ defined in (1.1), from \cite[Chapter III]{lady} with $f_i=g_i,\,i=1,\dots, N$. Note that $a$ is matrix valued. We have switched the notation for the unknown from $u$ to $v$ and for the initial data from $\psi_0$ to $v_0$. Condition (1.2) in \cite[Chapter III]{lady} is precisely the condition on $a$ just above.

\begin{theorem}\label{thm:lady01}
Assume that $v_0\in C^\beta(\overline D)$ for some $\beta>0$ and $v_0=0$ on $\partial D$.
There exists $r_0$ depending on $N$; $\alpha\in(0,\beta]$ depending on $N,\nu,\mu,r_0,D$ and a constant $K_1$ depending on $N,\nu,\mu,r_0,D,\alpha$ such that if
\begin{itemize}
\item[(i)] $v$ is a weak solution to \eqref{eq} which is continuous in time with values in $L^2$ and belongs to $L^2(0,T;W^{1,2}_0)$,
\item[(ii)] $f\in L^{r_0}(D_T)$ and $g\in L^{2r_0}(D_T)$,
\end{itemize}
then $v\in C^{\alpha/2,\alpha}(\overline{D_T})$ and the following estimate holds
\begin{equation}\label{eq:estimate}
\begin{split}
&\|v\|_{C^{\alpha/2,\alpha}}\leq K_1\big(\|v_0\|_{C^\alpha(\overline{D})}+\|g\|_{L^{2r_0}(D_T)}+\|f\|_{L^{r_0}(D_T)}\big).
\end{split}
\end{equation}

\begin{proof}
This result follows from \cite[Theorem 10.1, Chapter III]{lady}. 
Note the compatibility 
condition $v_0=0$ on $\partial D$ implies that the data on the parabolic boundary 
$\Gamma_T= \{0\}\times D\cup S_T$ is in  $C^{\alpha/2,\alpha}(\Gamma_T)$, which is denoted 
by $H^{\alpha,\alpha/2}(\Gamma_T)$ in \cite{lady}.

Numbers of equations and theorems below refer to \cite[Chapter III]{lady}.

We first have to check that $v$ is in $L^\infty([0,T]\times D)$. This follows from Theorem 7.1..  

We first note that it is always possible to find a number $r_0$ depending on $N$ such that the couple $(r,q)=(r_0,r_0)$ satisfies (7.2) with some $\kappa_1\in (0,1)$ if $N\geq 2$ or $\kappa_1\in (0,\frac{1}{2})$ if $N=1$. Hypothesis (7.1) with $\mu_1$ depending on $r_0$ and $N$ then follows.
The crucial point is the estimate (7.14) then Theorem 6.1 from Chapter II applies and yields the maximum principle. Moreover, the dependence of the $L^\infty$-bound on the coefficients and the initial data can be seen from (6.2). We obtain
\begin{align*}
\begin{aligned}
\|v\|_{L^\infty(D_T)}\leq C\big(1+\|v_0\|_{L^\infty}\big)\Big(1+\|g\|_{L^{2r_0}(D_T)}^{2(1+\frac{N}{2\kappa_1})}+\|f\|_{L^{r_0}(D_T)}^{1+\frac{N}{2\kappa_1}}\Big)
\end{aligned}
\end{align*}
where the constant $C$ does not depend on $v_0,\,g,\,f$ and depends on $a$ only through its ellipticity constant $\nu$. In other words we have proved that the linear mapping
\begin{align}\label{eq:solmap}
\begin{aligned}
L^\infty(D)\times L^{r_0}(D_T)\times L^{2r_0}(D_T)&\rightarrow L^\infty(D_T)\\
(v_0,f,g)&\mapsto v
\end{aligned}
\end{align}
is bounded hence continuous and it holds
\begin{align}\label{3}
\|v\|_{L^\infty(D_T)}\leq C\big(\|v_0\|_{L^\infty}+\|g\|_{L^{2r_0}(D_T)}+\|f\|_{L^{r_0}(D_T)}\big).
\end{align}

Since we assume zero Dirichlet boundary conditions we may now apply the second part of Theorem 10.1 and deduce that $v$ is in $C^{\alpha/2,\alpha}(\overline{D_T})$. The proof is based on estimate (10.6) applied to cylinders $Q(\rho,\tau)$ intersecting the parabolic boundary $\Gamma_T=\{0\}\times D\cup S_T$ and to levels $k\geq \max_{Q(\rho,\tau)\cap\Gamma_T}\pm u$ and follows from Theorem 8.1 in Chapter II. 
Note that due to Remark 7.1 in Chapter II, the H\"older exponent $\alpha$ depends neither on the $L^\infty$-bound of $v$ nor on the $L^{r_0}$- and $L^{2r_0}$-norm of $f$ and $g$, respectively, but may depend on the $C^\beta$-norm of $v_0$ as can be seen from the proof of Theorem 8.1, Chapter II, namely from (8.6). Nevertheless, one can first take $\|v_0\|_{C^\beta}\leq 1$ and then argue by linearity of the solution map \eqref{eq:solmap} to obtain the final estimate \eqref{eq:estimate}.
\end{proof}
\end{theorem}

As the next step, we recall the Schauder estimate for equations in divergence form \cite[Theorem 6.48]{lieberman}.

\begin{theorem}\label{thm:schauder}
Let $\alpha\in (0,1)$, assume that $D$ has a $C^{\alpha+1}$ boundary and 
\begin{itemize}
\item[(i)] $a, g\in C^{\alpha/2,\alpha}(\overline{D_T})$, $f\in L^p(D_T)$ for some $p\geq \frac{N+2}{1-\alpha}$,
\item[(ii)] $v_0\in C^{1+\alpha}(\overline{D})$, $v_0=0$ on $\partial D$.
\end{itemize}
Then there exists a unique weak solution to \eqref{eq}. Moreover, there exists a constant $K_2$ depending on $N,\nu,\mu,p,D,\alpha$ such that
\begin{align}\label{eq:schest}
\begin{aligned}
\|v\|_{C^{(1+\alpha)/2,1+\alpha}}\leq K_2 P_1(a_{\alpha/2,\alpha})\big(\|v_0\|_{C^{1+\alpha}(\overline D)}+\|g\|_{C^{\alpha/2,\alpha}}+\|f\|_{L^p(D_T)}\big),
\end{aligned}
\end{align}
where 
$P_1$ is a polynomial and $a_{\alpha/2,\alpha}=\|a\|_{C^{\alpha/2,\alpha}}$.

\begin{proof}
According to \cite[Theorem 6.48]{lieberman} one obtains the required regularity of $v$. Note that our compatibility condition $v_0=0$ on $\partial D$ implies that the boundary data are in $C^{(1+\alpha)/2,1+\alpha}$ on the parabolic boundary $S_T$. This parabolic boundary 
is denoted by ${\mathcal P}\Omega$ in \cite{lieberman} and this H\"older space by 
$H_{1+\alpha}$. Thus the compatibility condition of \cite[Theorem 6.48]{lieberman} is satisfied. 

But the dependence on $a_{{\alpha/2,\alpha}}$ is not obvious. In the following we adapt the technique from \cite[Theorem 3.1]{juliaarnaud} and show that the dependence is indeed polynomial.

Take $r\in(0,1]$ to be fixed later. Let $(B_i(r/4))_{i=1,\dots,n}$ be a covering of $\overline{D_T}$ by parabolic cylinders of radius $r$, i.e. balls of radius $r$ with respect to the parabolic distance
$$d((t,x),(s,y))=\max\{|t-s|^{1/2},|x-y|\}.$$
We denote their centers by $(t_i,x_i)$. Let
$$\varphi_i\in C^\infty_c(B_i(r))\text{ such that } 0\leq \varphi_i\leq 1\text{ and }\varphi_i\equiv 1 \text{ on } B_i(r/2).$$
It is possible to choose these functions such that
$$\|\nabla \varphi_i\|_{C^0(\overline{D_T})}\leq C r^{-1},\quad\|\partial_t\varphi_i\|_{C^0(\overline{D_T})}\leq C r^{-2},\quad\|\nabla\varphi_i\|_{C^{\alpha/2,\alpha}}\leq C r^{-1-\alpha}.$$
Set $a_i=a(t_i,x_i), \,v_i=\varphi_i v,\, v_{0,i}=\varphi_i(0) v_0$. Then $v_i$ solves the following parabolic equation
\begin{equation}\label{eq:const}
\left\{
\begin{aligned}
\partial_t v_i-\diver(a_i\nabla v_i)&=\diver[(a-a_i)\nabla v_i]+\varphi_i f+\diver(\varphi_i g)-g\nabla\varphi_i\\
&\qquad-a\nabla v\nabla\varphi_i-\diver[a v\nabla\varphi_i]+v\partial_t\varphi_i,&\text{ in }&D_T,\\
v_i&=0,&\text{ on }&S_T,\\
v_i(0)&=v_{0,i},&\text{ in }&D.
\end{aligned}
\right.
\end{equation}
This is an equation of the form
\begin{equation*}
\left\{
\begin{aligned}
\partial_t w-\diver(b\nabla w)&=h_0+\diver h,&\text{ in }&D_T,\\
w&=0,&\text{ on }&S_T,\\
w(0)&=w_{0},&\text{ in }&D,
\end{aligned}
\right.
\end{equation*}
where $b$ is a constant matrix. According to \cite[Theorem 6.48]{lieberman}, it holds true
\begin{align}\label{e3.8}
\begin{aligned}
\|w\|_{C^{(1+\alpha)/2,1+\alpha}}\leq C\big(\|w_0\|_{C^{1+\alpha}}+\|h\|_{C^{\alpha/2,\alpha}}+\|h_0\|_{M^{1,N+1+\alpha}(D_T)}\big),
\end{aligned}
\end{align}
where the constant depends only on $N,\nu,\mu,D,\alpha$ and $M^{1,N+1+\alpha}(D_T)$ denotes the Morrey space (corresponding to the parabolic distance). It is easy to see that the 
compatibility conditions are satisfied by \eqref{eq:const}.

By H\"older inequality it follows from the definition of the norm in $M^{1,N+1+\alpha}(D_T)$ that if $p\geq\frac{N+2}{1-\alpha}$ then $L^p(D_T)\hookrightarrow M^{1,N+1+\alpha}(D_T)$. Therefore
\begin{align}\label{1}
\begin{aligned}
\|w\|_{C^{(1+\alpha)/2,1+\alpha}}\leq C_0\big(\|w_0\|_{C^{1+\alpha}}+\|h\|_{C^{\alpha/2,\alpha}}+\|h_0\|_{L^{p}(D_T)}\big)
\end{aligned}
\end{align}
with a constant depending only on $N,\nu,\mu,D,\alpha$.

We apply this bound to \eqref{eq:const} and obtain
\begin{align}\label{2}
\begin{aligned}
\|v_i\|_{C^{(1+\alpha)/2,1+\alpha}}&\leq C_0\Big[\|\varphi_i(0)\|_{C^{1+\alpha}}\|v_{0}\|_{C^{1+\alpha}}+\|(a-a_i)\nabla v_i\|_{C^{\alpha/2,\alpha}}\\
&\qquad\quad+\|f\|_{L^p(D_T)}+\|\varphi_i \|_{C^{\alpha/2,\alpha}}\|g \|_{C^{\alpha/2,\alpha}}\\
&\qquad\quad+\|\nabla\varphi_i\|_{C^0(\overline{D_T})}\|g\|_{L^p(D_T)}\\
&\qquad\quad+\|a\|_{C^0(\overline{D_T})}\|\nabla\varphi_i\|_{C^0(\overline{D_T})}\|\nabla v\|_{L^p(D_T)}\\
&\qquad\quad+\|a\|_{C^{\alpha/2,\alpha}}\|v \|_{C^{\alpha/2,\alpha}}\|\nabla\varphi_i \|_{C^{\alpha/2,\alpha}}\\
&\qquad\quad+\|v\|_{C^0(\overline{D_T})}\|\partial_t\varphi_i\|_{C^0(\overline{D_T})}\Big].
\end{aligned}
\end{align}
In the following, $|\cdot|_{\alpha/2,\alpha}$ denotes the $\alpha$-H\"older seminorm with respect to the parabolic distance, i.e.
$$|w|_{\alpha/2,\alpha}=\sup_{\substack{(t,x),(s,x)\in D_T\\(t,x)\neq(s,x)}}\frac{|w(t,x)-w(s,x)|}{d((t,x),(s,y))^\alpha}.$$
By interpolation, there exists $C>0$ such that we have for every $\varepsilon>0$ (see e.g. \cite[Lemma 3.2, Chapter II]{lady}):
$$\| v_i\|_{C^{1/2,1}}\leq \varepsilon\| v_i\|_{C^{(1+\alpha)/2,1+\alpha}}+\frac{C}{\varepsilon^{1/\alpha}}\|v_i\|_{C^0(\overline{D_T})}.$$
Besides,
\begin{align*}
\|\nabla v_i\|_{C^0(\overline{D_T})}&\leq \| v_i\|_{C^{1/2,1}}\\
\|\nabla v_i\|_{C^0(\overline{D_T})}+|\nabla v_i|_{\alpha/2,\alpha}&\leq \|v\|_{C^{(1+\alpha)/2,1+\alpha}}
\end{align*}
hence choosing $\varepsilon=r^\alpha$ we obtain
\begin{align*}
\|(a-a_i)\nabla v_i\|_{C^{\alpha/2,\alpha}}&\leq\|a-a_i\|_{C^0(\overline{B_i(r)})}\|\nabla v_i\|_{C^0(\overline{D_T})}\\
&\quad +\|a-a_i\|_{C^0(\overline{B_i(r)})}|\nabla v_i|_{\alpha/2,\alpha}+a_{\alpha/2,\alpha}\|\nabla v_i\|_{C^0(\overline{D_T})}\\
&\leq a_{\alpha/2,\alpha} r^\alpha\|\nabla v_i\|_{C^0(\overline{D_T})}+a_{\alpha/2,\alpha} r^\alpha|\nabla v_i|_{\alpha/2,\alpha}\\
&\quad+a_{\alpha/2,\alpha}\|\nabla v_i\|_{C^0(\overline{D_T})}\\
&\leq a_{\alpha/2,\alpha}\big(2r^\alpha\| v_i\|_{C^{(1+\alpha)/2,1+\alpha}}+Cr^{-1}\|v_i\|_{C^0(\overline{D_T})}\big).
\end{align*}
We set $r=\big(4C_0 a_{\alpha/2,\alpha}\big)^{-1/\alpha}\wedge 1$, where $C_0$ is the constant defined in \eqref{1}, and deduce from \eqref{2}
\begin{align*}
\|v_i\|_{C^{(1+\alpha)/2,1+\alpha}}&\leq C\Big[r^{-1-\alpha}\|v_{0}\|_{C^{1+\alpha}}+a_{\alpha/2,\alpha}r^{-1}\|v_i\|_{C^0(\overline{D_T})}+\|f\|_{L^p(D_T)}\\
&\qquad\quad+r^{-1}\|g \|_{C^{\alpha/2,\alpha}}+a_{\alpha/2,\alpha}r^{-1}\|\nabla v\|_{L^p(D_T)}\\
&\qquad\quad+a_{\alpha/2,\alpha}r^{-1-\alpha}\|v \|_{C^{\alpha/2,\alpha}}+r^{-2}\|v\|_{C^0(\overline{D_T})}\Big].
\end{align*}
Next, we estimate the second, the sixth and the seventh term on the right hand side by \eqref{eq:estimate} and \eqref{3} and obtain
\begin{align}\label{4}
\begin{aligned}
\|v_i&\|_{C^{(1+\alpha)/2,1+\alpha}}\leq C r^{-1-\alpha}\|\nabla v\|_{L^p(D_T)}\\
&+ C\big(r^{-1-2\alpha}+r^{-2}\big)\big(\|v_{0}\|_{C^{1+\alpha}}+\|f\|_{L^p(D_T)}+\|g \|_{C^{\alpha/2,\alpha}}\big).
\end{aligned}
\end{align}
Since $p\geq 2$ we also have the elementary interpolation inequality
$$\|\nabla v\|_{L^p({D_T})}^p\leq\|\nabla v\|_{C^0(\overline{D_T})}^{p-2}\|\nabla v\|_{L^2({D_T})}^2$$
as well as the basic energy estimate
\begin{equation}\label{44}
\|v\|_{L^\infty(0,T;L^2)}+\|\nabla v\|_{L^2(D_T)}\leq C\big(\|v_0\|_{L^2}+\|g\|_{L^2(D_T)}+\|f\|_{L^2(D_T)}\big)
\end{equation}
which holds true with a constant that depends on $\nu$ but otherwise is independent of the coefficients of \eqref{eq}.
We deduce
\begin{align*}
\|\nabla v\|_{L^p(\overline{D_T})}^p&\leq C \|\nabla v\|_{C^0(\overline{D_T})}^{p-2}\big(\|v_{0}\|_{C^{1+\alpha}}+\|f\|_{L^p(D_T)}+\|g \|_{C^{\alpha/2,\alpha}}\big)^2
\end{align*}
Therefore, Young's inequality yields
\begin{align*}
C &r^{-1-\alpha}\|\nabla v\|_{L^p(D_T)}\\
&\leq C r^{-1-\alpha}\|\nabla v\|_{C^0(\overline{D_T})}^{\frac{p-2}{p}}\big(\|v_{0}\|_{C^{1+\alpha}}+\|f\|_{L^p(D_T)}+\|g \|_{C^{\alpha/2,\alpha}}\big)^\frac{2}{p}\\
&\leq Cr^{-\frac{p}{2}(1+\alpha)}\big(\|v_{0}\|_{C^{1+\alpha}}+\|f\|_{L^p(D_T)}+\|g \|_{C^{\alpha/2,\alpha}}\big)+\frac{1}{2}\|\nabla v\|_{C^0(\overline{D_T})}.
\end{align*}
Hence it follows from \eqref{4}
\begin{align}\label{5}
\begin{aligned}
\|&v_i\|_{C^{(1+\alpha)/2,1+\alpha}}\leq \frac{1}{2} \|\nabla v\|_{C^0(\overline{D_T})}\\
&+ C\big(r^{-1-2\alpha}+r^{-2}+r^{-\frac{p}{2}(1+\alpha)}\big)\big(\|v_{0}\|_{C^{1+\alpha}}+\|f\|_{L^p(D_T)}+\|g \|_{C^{\alpha/2,\alpha}}\big).
\end{aligned}
\end{align}
Let $(t_0,x_0)\in\overline{D_T}$ such that $|\nabla v(t_0,x_0)|=\|\nabla v\|_{C^0(\overline{D_T})}$ and take $i_0$ such that $(t_0,x_0)\in B_{i_0}(r/4)$. Then $\nabla v(t_0,x_0)=\nabla v_{i_0}(t_0,x_0)$. Therefore
$$\|\nabla v\|_{C^0(\overline{D_T})}\leq \|\nabla v_{i_0}\|_{C^0(\overline{D_T})}$$
and consequently
\begin{align}\label{6}
\begin{aligned}
\|&\nabla v\|_{C^0(\overline{D_T})}\\
&\leq C\big(r^{-1-2\alpha}+r^{-2}+r^{-\frac{p}{2}(1+\alpha)}\big)\big(\|v_{0}\|_{C^{1+\alpha}}+\|f\|_{L^p(D_T)}+\|g \|_{C^{\alpha/2,\alpha}}\big).
\end{aligned}
\end{align}
Plugging this back in \eqref{5} we obtain
\begin{align*}
\begin{aligned}
\|&v_{i}\|_{C^{(1+\alpha)/2,1+\alpha}}\\
&\leq C\big(r^{-1-2\alpha}+r^{-2}+r^{-\frac{p}{2}(1+\alpha)}\big)\big(\|v_{0}\|_{C^{1+\alpha}}+\|f\|_{L^p(D_T)}+\|g \|_{C^{\alpha/2,\alpha}}\big).
\end{aligned}
\end{align*}
In order to estimate
$$\|v\|_{(1+\alpha)/2,1+\alpha}=\|v\|_{C^0(\overline{D_T})}+\|\nabla v\|_{C^0(\overline{D_T})}+|\nabla v|_{\alpha/2,\alpha},$$
we apply \eqref{3}, \eqref{6} and the inequality
$$|\nabla v|_{\alpha/2,\alpha}\leq \max\big\{\max_i|\nabla v_i|_{{\alpha/2,\alpha}},2^{1+2\alpha}r^{-\alpha}\| \nabla v\|_{C^0(\overline{D_T})}\big\}
$$
which is easily obtained by considering the points where the maximum in the definition of $|\nabla v|_{\alpha/2,\alpha}$ is reached. 
Indeed, if $(t,x)$ and $(s,y)$ are such points, then $|\nabla v|_{\alpha/2,\alpha}$ is bounded by
$$2^{1+2\alpha}r^{-\alpha}\|\nabla v\|_{C^0(\overline{D_T})}\quad\text{provided}\quad d((t,x),(s,x))\geq \frac{r}{4}$$
and by $|\nabla v_i|_{\alpha/2,\alpha}$ otherwise, where $i$ is such that $(t,x),(s,y)\in B_i(r/2)$. We deduce
\begin{align*}
\|&v\|_{(1+\alpha)/2,1+\alpha}\\
&\leq C(r^{-1-3\alpha}+r^{-2-\alpha}+r^{-\frac{p+\alpha p+2\alpha}{2}}\big)\big(\|v_{0}\|_{C^{1+\alpha}}+\|f\|_{L^p(D_T)}+\|g \|_{C^{\alpha/2,\alpha}}\big)
\end{align*}
which implies \eqref{eq:schest}.
\end{proof}
\end{theorem}


We now give higher regularity results for equation in non divergence form. We consider a linear parabolic PDE of the form
\begin{equation}\label{eq-bis}
\left\{
\begin{aligned}
\partial_t v&=\sum_{i,j=1}^N a_{ij}(t,x)\partial_{i,j} v+f(t,x),&\text{ in }&D_T,\\
v&=0,&\text{ on }&S_T,\\
v(0)&=v_0,&\text{ in }&D,
\end{aligned}
\right.
\end{equation}
and assume that the matrix $a(t,x)$ is symmetric and that there exist $\nu,\mu>0$ such that $\nu|\xi|^2\leq a(t,x)\xi\cdot\xi\leq\mu|\xi|^2$, for all $(t,x)\in D_T$, $\xi\in\mr^N$.

We use a similar approach to obtain regularity as for the proof of Theorem \ref{thm:schauder}. 
In particular, this enables us  to quantify the constant with respect to the
parameters. 
We define $\cA$ as the second order  differential operator: $\cA w= \sum_{i,j=1}^N a_{i,j} \partial_{i,j}w$.

 
 We first recall low regularity in the parabolic scaling.
  \begin{theorem}\label{t3.3}
 Let $\alpha \in(0,1)$, $k=0,1$ and assume that $D$ has a $C^{\alpha+k+2}$ boundary. Consider the equation:
\begin{equation}\label{e3.17}
\left\{
\begin{aligned}
\partial_t v-\cA v &= f,&\text{ in }&D_T,\\
v&={\phi},&\text{ on }&S_T,\\
v(0)&=v_{0},&\text{ in }&D.
\end{aligned}
\right.
\end{equation}
Assume that 
\begin{itemize}
\item[(i)]  $f\in C^{(\alpha+k)/2,\alpha+k}( \overline{D_T})$, $a\in C^{(\alpha+k)/2,\alpha+k}(\overline{D_T})$.
\item[(ii)] $v_0\in C^{\alpha+k+2}(\overline D)$, 
 \item[(iii)] $\phi \in C^{(\alpha+k)/2+1,\alpha+k+2}(\overline{S_T})$ 
 \item[(iv)] $v_0=\phi $, $\partial_t\phi -\cA v_0=f$ on $\{0\}\times \partial D$.
\end{itemize}
Then there exists a unique solution $v\in C^{(\alpha+k)/2+1, \alpha+k+2}(\overline{D_T})$ and
\begin{equation}\label{truc}
\begin{split}
\|v\|_{ C^{(\alpha+k)/2+1, \alpha+k+2}} &\le 
K_3 P_2(A_{\alpha+k})\big(\|f\|_{C^{(\alpha+k)/2,\alpha+k}}+\|v_0\|_{C^{\alpha+k+2}(\overline D)}\\
&\; +\|\phi\|_{C^{(\alpha+k)/2+1,\alpha+k+2}(\overline{S_T})}\big),
\end{split}
\end{equation}
where $K_3$ depends only on $\mu$, $\nu$, $D$, $\alpha$, $P_2$ is a polynomial and 
$A_{\alpha+k}=\|a\|_{C^{(\alpha+k)/2,\alpha+k}}$.
\end{theorem}
\begin{proof}According to   Theorem 5.2, chapter IV in \cite{lady}, one obtains the required regularity of $v$.
 Indeed, $H^{l,l/2}$ in \cite{lady} corresponds to $C^{l/2,l}$ in our 
 notations and (iv) is the compatibility condition of order one. 
But as in Theorem \ref{thm:schauder}, the dependence on $A_{\alpha+k}$ is not obvious, we proceed 
similarly.

Take $r\in(0,1]$ to be fixed. Let $(B_i(r/4))_{i=1,\dots,n}$ be a covering of $\overline{D_T}$ by parabolic cylinders of radius $r$, i.e. balls of radius $r$ with respect to the parabolic distance,
and center $(t_i,x_i)$. Let
$$\varphi_i\in C^\infty_c(B_i(r))\text{ such that } 0\leq \varphi_i\leq 1\text{ and }\varphi_i\equiv 1 \text{ on } B_i(r/2),$$
such that
$$\|\nabla \varphi_i\|_{C^0(\overline{D_T})}\leq C r^{-1},\quad\|\partial_t\varphi_i\|_{C^0(\overline{D_T})}\leq C r^{-2},$$
$$\|\nabla\varphi_i\|_{C^{\alpha/2,\alpha}}\leq C r^{-1-\alpha},\quad\|\partial^\gamma \varphi_i\|_{C^{\alpha/2,\alpha}}\leq C r^{-2-\alpha},$$
for any multi-index $\gamma$ of length $|\gamma|=2$.

Set $\cA_i=\cA(t_i,x_i), \,v_i=\varphi_i v,\, v_{0,i}=\varphi_i(0) v_0$, $\phi_i=\varphi_i\phi$. Then $v_i$ solves the following parabolic equation
\begin{equation}\label{eq:const2}
\left\{
\begin{aligned}
\partial_t v_i-  \cA_i v_i&=v\partial_t \varphi_i+ (\cA-\cA_i)v_i - v\cA \varphi_i \\
&\quad-2\sum_{l,m=1}^N a_{lm}\partial_l v\partial_m \varphi_i+ f\varphi_i,&\text{ in }&D_T,\\
v_i&=\phi_i,&\text{ on }&S_T,\\
v_i(0)&=v_{0,i},&\text{ in }&D.
\end{aligned}
\right.
\end{equation}
This equation is similar to \eqref{e3.17} but the operator has constant coefficients. It is easy to see
that it satisfies the compatibility condition of order one.
Thus, according to Theorem 5.2, chapter IV in \cite{lady}, we have
\begin{align*}
\begin{aligned}
\|v_i\|_{C^{(\alpha+k)/2+1,\alpha+k+2}}&\leq C\big(\|v_{0,i}\|_{C^{\alpha+k+2}(\overline D)}+\|\phi_i\|_{C^{(\alpha+k)/2+1,\alpha+k+2}(\overline{S_T})}\\
&\quad+\|v\partial_t \varphi_i+ (\cA-\cA_i)v_i - v\cA \varphi_i \\
&\quad-2\sum_{l,m=1}^N a_{lm}\partial_l v\partial_m \varphi_i+ f\varphi_i\|_{C^{(\alpha+k)/2,\alpha+k}}\big),
\end{aligned}
\end{align*}
where the constant depends only on $N,\nu,\mu,D,\alpha$. To see this, one follows the proof of Theorem 5.2, Chapter IV in \cite{lady} for the case of constant coefficients. First, the equation is transformed by a local change of coordinates into a heat equation on the half space. In the case of zero initial data (see the problem (5.3)' in Chapter IV in \cite{lady}) one obtains the above estimate in Theorem 6.1, Chapter IV in \cite{lady}. Reduction of the problem to the problem with zero initial data is based on Theorem 4.3, Chapter IV in \cite{lady} and the method is explained in Section 8, Chapter IV in \cite{lady}. 

The end of the proof is exactly as above for 
Theorem \ref{thm:schauder} and uses the maximum principle for parabolic equations in non divergence form, see e.g. \cite[Theorem 8, Chapter 2]{friedman}.
\end{proof}
%

We now study  higher regularity. Since the application we have in mind is SPDEs, we do not have 
more than $1/2$ regularity in time. Hence, to increase the spatial regularity, we cannot use the 
parabolic scaling anymore. Unfortunately, all the results in the literature are in the parabolic 
scaling. To avoid lengthy technical proofs, we only investigate the first step: spatial regularity of 
order $k+\alpha$, $k= 4$, $\alpha <1$. Similar arguments can be used to get higher regularity. 

\begin{theorem}\label{t3.4}
Let $\alpha \in(0,1)$,  and assume that $D$ has a $C^{\alpha+4}$-boundary. Let $v$ be the solution
of 
\begin{equation*}
\left\{
\begin{aligned}
\partial_t v-\cA v &= f,&\text{ in }&D_T,\\
v&=0,&\text{ on }&S_T,\\
v(0)&=v_{0},&\text{ in }& D.
\end{aligned}
\right.
\end{equation*}
Assume that 
\begin{itemize}
\item[(i)]  $f\in C^{\alpha/2,\alpha+2}( \overline{D_T})$, $f|_{\overline{S_T}}\in C^{\alpha/2+1,\alpha+2}(\overline{S_T})$,   $a\in C^{\alpha/2,\alpha+2}( \overline{D_T})$, and  $a|_{\overline{S_T}}\in C^{\alpha/2+1,\alpha+2}( \overline{S_T})$.
\item[(ii)] $v_0\in C^{\alpha+4}(\overline D)$, 
 \item[(iii)] $v_0=0$, $\cA v_0+f=0$, and  $\partial_t f +\cA^2v_0+\cA f+ (\partial_t \cA )v_0=0$, on 
 $\{0\}\times \partial D$.

\end{itemize}
then $v\in C^{\alpha/2+1, \alpha+4}(\overline{D_T})$ and
\begin{equation}\label{est345}
\begin{split}
\|v\|_{ C^{\alpha/2+1, \alpha+4}} &\le
K_4 P_3(A_{\alpha+2}) \big(\|f\|_{C^{\alpha/2,\alpha+2}}\\
&\quad+ \|f|_{{S_T}}\|_{ C^{\alpha/2+1,\alpha+2}(\overline{S_T})}+\|v_0\|_{C^{\alpha+4}(\overline D)}\big),
\end{split}
\end{equation}
where $P_3$ is a polynomial, $A_{\alpha+2}=\|a\|_{C^{\alpha/2,\alpha+2}}
+\|a|_{\overline{S_T}}\|_{ C^{\alpha/2+1,\alpha+2}( \overline{S_T})}$ and $K_4$ depends only on $\alpha$, $\nu$, $\mu$, $D$.
\end{theorem}
\begin{remark}
It might seem strange that we need to assume more time regularity for $f$ on the boundary 
$S_T$. However, it is immediate that a solution in $C^{\alpha/2+1, \alpha+4}(\overline D_T)$ satisfies,  
$-\cA v = f$ on $S_T$. Thus, $f|_{S_T}\in C^{\alpha/2+1,\alpha+2}(\overline{S_T})$. Concerning the compatibility 
conditions (iii), the first two are already in Theorem \ref{t3.3}. Since the equation holds up to the boundary, we have $-\cA v=f$ on $\overline{S_T}$. Also, the solution is sufficiently regular to differentiate it and get $\partial_t (\cA v) -(\partial_t \cA) v  -\cA^2 v =\cA f$ so that on $\overline{S_T}$:
$- \partial_t f -(\partial_t \cA) v  -\cA^2 v =\cA f$ and for $t=0$ we see that the third condition is also
necessary.

Note that this third compatibility condition is exactly the classical second order condition required 
to get smoothness in the parabolic scaling (see \cite{lady}, section 5, chapter IV. In particular (5.6) and below).
\end{remark}

 We set in the following result: $\R^{N,+}=\{x\in \R^N:\; x_N\ge 0\}$, $\R^{N,++}=\{x\in \R^N:\; x_N>0\}$. Note that its boundary  is $\R^{N-1}$. As usual, for $x=(x_1,\dots,x_N)\in \R^N$, we set
 $x'=(x_1,\dots,x_{N-1})$. Also, given a function $g$ defined on $\R^{N,+}$, we denote by $g|_{x_N=0}$ its restriction to $\{x_N=0\}$. 
 
 The proof of Theorem \ref{t3.4} follows the lines of the proof of Theorem 5.2, chapter IV in \cite{lady}. 
It is more complicated than  the proof of Theorem \ref{t3.3} since we cannot base our argument on 
known results. We start with the case of the Laplace operator on the half space. The following Lemma 
provides the missing ingredient to reproduce the argument of \cite{lady} which is based on inequalities 
(2.1), (2.2), (2.3) from Chapter IV.  Inequality \eqref{est346} is the generalization of these to indices that are not 
in the parabolic scaling.
  
\begin{lemma}\label{l3.4}
Let $\alpha \in(0,1)$ and $v$ be the solution
of 
\begin{equation*}
\left\{
\begin{aligned}
\partial_t v-\Delta v &= f,&\text{ in }&\R^{N,++}\times (0,T],\\
v&=0,&\text{ on }&\R^{N-1}\times (0,T],\\
v(0)&=v_{0},&\text{ in }&\R^{N,++}.
\end{aligned}
\right.
\end{equation*}
Assume that 
\begin{itemize}
\item[(i)]  $f\in C^{\alpha/2,\alpha+2}( \R^{N,+} \times [0,T])$  and $f|_{x_N=0}\in C^{\alpha/2+1, \alpha+2}( \R^{N-1} \times [0,T])$,
\item[(ii)] $v_0\in C^{\alpha+4}(\R^{N,+})$, 
 \item[(iii)] $v_0=0$, $-\Delta v_0=f$ and $\partial_tf+\Delta^{2}v_0+\Delta f=0$ for $x_N=0,\; t=0$.

\end{itemize}
then $v\in C^{\alpha/2+1, \alpha+4}(\R^{N,+} \times [0,T])$ and
\begin{equation}\label{est346}
\begin{split}
\|v\|_{ C^{\alpha/2+1, \alpha+4}(\R^{N,+} \times [0,T])} &\le
K_5 \big(\|f\|_{C^{\alpha/2,\alpha+2}( \R^{N,+} \times [0,T])} \\
&+ \| f\|_{ C^{\alpha/2+1, \alpha+2}( \R^{N-1} \times [0,T])}\\
&+\|v_0\|_{C^{\alpha+4}(\R^{N,+})}\big),
\end{split}
\end{equation}
where $K_5$ depends only on $\alpha$ and $N$.
\end{lemma}
\begin{proof}

We already know from Theorem \ref{t3.3} that $v\in C^{\alpha/2+1, \alpha+2}(\R^{N,+} \times [0,T])$.
 
Let $w=\partial_{ij}v$, with $i\ne N, \, j\ne N$. It satisfies 
\begin{equation*}
\left\{
\begin{aligned}
\partial_t w-\Delta w &= \partial_{ij} f,&\text{ in }&\R^{N,++}\times (0,T],\\
w&=0,&\text{ on }&\R^{N-1}\times (0,T],\\
w(0)&=\partial_{ij} v_{0},&\text{ in }&\R^{N,++}.
\end{aligned}
\right.
\end{equation*}
Clearly, $\partial_{ij}f$ and $\partial_{ij}v_0$ satisfy the assumption of Theorem \ref{t3.3}. We deduce that 
$\partial_{ij}v\in C^{\alpha/2+1, \alpha+2}(\R^{N,+} \times [0,T])$.

It remains to prove that $\partial_n^{2}v\in C^{1+\alpha/2, \alpha+2}(\R^{N,+} \times [0,T])$. 
We write $w=\partial_n^{2}v$. Note 
that:
$$
w=\partial_t v-\Delta_{x'} v -f.
$$
Hence taking the restriction at $x_N=0$:
$$
w|_{x_N=0}=-f.
$$
We deduce that $w$ satisties:
\begin{equation*}
\left\{
\begin{aligned}
\partial_t w-\Delta w &= \partial_n^{2}f,&\text{ in }&\R^{N,++}\times (0,T],\\
w&=-f,&\text{ on }&\R^{N-1}\times (0,T],\\
w(0)&=\partial_n^{2} v_{0},&\text{ in }&\R^{N,++}.
\end{aligned}
\right.
\end{equation*}
Our assumptions imply that $\partial_n^2f$, $\partial_n^{2}v_0$ and $f|_{S_T}$ 
satisfy the assumptions of Theorem \ref{t3.3}. We deduce that $w\in C^{\alpha/2+1, \alpha+2}(\R^{N,+} \times [0,T])$ and this finishes the proof. Note that \eqref{est346} follows from the estimate 
given on $\partial_{ij}v$ and $\partial_n^{2}v$ by Theorem \ref{t3.3}.
\end{proof}

Following the argument in section 6, chapter IV in \cite{lady}, we generalize Lemma \ref{l3.4} to 
the case of a general elliptic operator with constant coefficient. 
\begin{lemma}\label{l3.5}
Let $A=\sum_{ij} \bar a_{ij}\partial_{ij}$ is a second order partial differential operator with constant 
coefficients such that  
 $\nu|\xi|^2\leq \bar a\xi\cdot\xi\leq\mu|\xi|^2$, for all $\xi\in \R^N$.
Let  $\alpha \in(0,1)$ and $v$ be the solution of 
\begin{equation*}
\left\{
\begin{aligned}
\partial_t v-A v &= f,&\text{ in }&\R^{N,++}\times (0,T],\\
v&=0,&\text{ on }&\R^{N-1}\times (0,T],\\
v(0)&=v_{0},&\text{ in }&\R^{N,++}.
\end{aligned}
\right.
\end{equation*}
Assume that 
\begin{itemize}
\item[(i)]  $f\in C^{\alpha/2,\alpha+2}(\R^{N,+}\times [0,T])$ and $f|_{ x_N=0}\in C^{\alpha/2+1, \alpha+2}( \R^{N-1} \times [0,T])$
\item[(ii)] $v_0\in C^{\alpha+4}(\R^{N,+})$, 
$-A v_0=f$  and $\partial_tf+A^{2}v_0+A f=0$ for $x_N=0,\; t=0$.
\end{itemize}
then $v\in C^{\alpha/2+1, \alpha+4}(D_T)$ and 
\begin{equation*}
\begin{split}
 \|v\|_{ C^{\alpha/2+1, \alpha+4}(\R^{N,+}\times [0,T])} &\le 
K_6 
\big(\|f\|_{C^{\alpha/2,\alpha+2}(\R^{N,+}\times [0,T])}+\| f|_{x_N=0}\|_{ C^{\alpha/2+1, \alpha+2}( \R^{N-1} \times [0,T])}\\
&\quad+\|v_0\|_{C^{\alpha+4}(\R^{N,+})}\big),
\end{split}
\end{equation*}
where $K_6$ depends on  $\alpha$, $N$, $\mu$ and $\nu$.
\end{lemma}

The proof of  Theorem \ref{t3.4}. Follows now exactly the proof given in section 7, chapter IV in \cite{lady}. 
It is quite long and we do not reproduce it here. 

Note contrary to Theorem \ref{t3.3}, we do not know
that $v$ has the required regularity. This has to be proved. As above, the proof  consists in  using 
a suitable covering of the domain $D_T$ into small sub-domains. On each of these domains 
intersecting the boundary, the  equation is transformed by a local change of coordinate into an 
equation on the half space. This latter equation is approximated by the same equation with frozen 
coefficients, whose solution can be estimated thanks to  Lemma \ref{l3.5}. If the sub-domain does not
intersect the boundary, an estimate is trivially obtained by differentiation of the equation. 
Putting together all these
local inverses, we show that we obtain an approximate inverse. Than an elementary argument allows to conclude. 

Note that contrary to \cite{lady}, we work with non zero initial data (see the discussion after Theorem 
5.3 for the definition of this notion). The price to pay is that we need to pay particular attention to the 
compatibility conditions but this does not cause any problem.

The polynomial dependence on $A_{\alpha+2}$ comes from similar computations as in the proof of Theorem \ref{t3.3}. In the notation of section 7, chapter IV in \cite{lady}, this comes from the choice
of $\tau$. It has to be chosen small enough so that $\|T\|<1/2$ and the proof clearly shows that this
condition can be written in terms of $A_{\alpha+2}$ and the local coordinate systems on $\partial D$.

\section{First step in the regularity problem: proof of Theorem \ref{mainresult1}}
\label{sec:first}

In this section, we show the first step towards regularity of the weak solution $u$ to \eqref{equation}.
We consider the auxiliary problem \eqref{aux123} with $\Psi=H(u)$,
whose solution is given by the stochastic convolution
\begin{equation}\label{defz}
z(t)=\int_0^tS(t-s)H(u_s)\,\dd W_s,\quad t\in[0,T].
\end{equation}
Next, we define the process $y:=u-z$. It follows immediately that $y$ solves the following linear parabolic PDE with random coefficients
\begin{equation}\label{diff}
\left\{
\begin{aligned}
\partial_t y&=\mathrm{div}\left(A(u)\nabla y\right)+\diver(B(u))+F(u)+\mathrm{div}\left((A(u)-\text{I})\nabla z\right),&\text{ in }&D_T,\\
y&= 0, &\text{ on }&S_T,\\
y(0)&= u_0,&\text{ in }&D.\\
\end{aligned}
\right.
\end{equation}
This way, we have split $u$ into two parts, i.e. $y$ and $z$, that are much more convenient in order to study regularity.
%
%
%

\begin{proof}[Proof of Theorem \ref{mainresult1}]

{\em Step 1: Regularity of $z$.} According to the hypothesis, the weak solution $u$ to $(\ref{equation})$ belongs to $L^2(\Omega;L^2(0,T;H^{1,2}_0))$ so that, thanks to the hypothesis (\hyperref[compgamma]{$\text{H}_{1,2}$}), we have that $H(u)$ belongs to $L^2(\Omega;L^2(0,T;\gamma(K,H^{1,2}_0)))$. As a result, with Proposition \ref{RegulStochastic} - $(ii)$ and the bound (\hyperref[compgamma]{$\text{H}_{1,2}$})
we have that for any $a\in(0,2)$, ${z}\in L^2(\Omega;L^2(0,T;H^{a,2}_0))$ with
\begin{equation}\label{eq:zz}
\E\|{z}\|^2_{L^2(0,T;H_0^{a,2})}\leq C\Big{(}1+\E\|u\|^2_{L^2(0,T;H^{1,2}_0)}\Big{)}
\end{equation}
and by Proposition \ref{RegulStochastic} - $(i)$, for $p>2$,
\begin{equation}\label{eq:zzz}
\E\|{z}\|^p_{C([0,T];L^{2})}\leq C\Big{(}1+\E\|u\|^p_{L^p(0,T;L^{2})}\Big{)}.
\end{equation}
Besides, since for all $p\in[2,\infty)$, the weak solution $u$ to $(\ref{equation})$ belongs to $L^p(\Omega;L^p(0,T;L^p))$, we obtain, with the hypothesis (\hyperref[compgamma]{$\text{H}_{0,p}$}) (see Remark \ref{H0r}), that $H(u)$ belongs to $L^p(\Omega;L^p(0,T;\gamma(K,L^p)))$. As a consequence, with Proposition \ref{RegulStochastic} - $(ii)$ and the bound (\hyperref[compgamma]{$\text{H}_{0,p}$})
we have that for any $b\in(0,1)$, ${z}\in L^p(\Omega;L^p(0,T;H^{b,p}_0))$ with
$$
\E\|{z}\|^p_{L^p(0,T;H^{b,p}_0)}\leq C\Big{(}1+\E\|u\|^p_{L^p(0,T;L^p)}\Big{)}.
$$
We have proved that for any $a\in(0,2)$ and $b\in(0,1)$, we have ${z}\in L^2(\Omega;L^2(0,T;H_0^{a,2}))$ and ${z}\in L^p(\Omega;L^p(0,T;H^{b,p}_0))$. We can now interpolate to obtain that (see \cite{amann})
$${z}\in L^r(\Omega;L^r(0,T;H_0^{c,r})),$$
where, for $\theta\in(0,1)$,
$$
\left\{
\begin{aligned}
\frac{1}{r}&=\frac{\theta}{2}+\frac{1-\theta}{p}, \\
c &= a\theta+b(1-\theta),
\end{aligned}
\right.
$$
with the bound
\begin{equation}\label{interpol}
\E\|{z}\|^r_{L^r(0,T;H_0^{c,r})}\leq \left(\E\|{z}\|^2_{L^2(0,T;H^{a,2}_0)}\right)^{r\theta/2}\left(\E\|{z}\|^p_{L^p(0,T;H^{b,p}_0)}\right)^{r(1-\theta)/p}<\infty.
\end{equation}
Note that by choosing $\theta\in(0,1)$ and $p\in[2,\infty)$ appropriately, $r$ can be arbitrary in $[2,\infty)$. Furthermore, when $\theta\in(0,1)$ is fixed, it is always possible to take $(a,b)\in(0,2)\times(0,1)$ such that $c>1$. As a result, for all $r\in[2,\infty)$, there exists $c_r>1$ such that
$${z}\in L^r(\Omega;L^r(0,T;H_0^{c_r,r})).$$
This gives, for all $r\in[2,\infty)$,
$$\nabla {z}\in L^r(\Omega;L^r(0,T;L^r)),$$
and, due to the boundedness of the mapping $A$,
$$(A(u)-\text{I})\nabla {z}\in L^r(\Omega;L^r(0,T;L^r)),$$
with, thanks to $(\ref{interpol})$,
\begin{equation}\label{ladyzenbound1}
\E\|(A(u)-\text{I})\nabla {z}\|^r_{L^r(0,T;L^r)}\leq C\E\|{z}\|^r_{L^r(0,T;H_0^{c,r})}<\infty,
\end{equation}
where $C>0$ depends on $\mu$ from \eqref{elliptic}.
Note that, thanks to the linear growth property of the coefficients $B$ and $F$, we obviously have, for all $r\in[2,\infty)$,
\begin{equation}\label{ladyzenbound2}
\E\|B(u)\|^r_{L^r(0,T;L^r)}+\,\E\|F(u)\|^r_{L^r(0,T;L^r)}\leq C(1+\E\|u\|^r_{L^r(0,T;L^r)})<\infty.
\end{equation}

{\em Step 2: Regularity of $y$.} 
We apply Theorem \ref{thm:lady01} with
$$a=A(u),\quad g=B(u)+((A(u)-\textrm{I})\nabla z),\quad f=F(u).$$
The assumptions are satisfied. Indeed \eqref{elliptic} gives the required uniform ellipticity
and boundedness of $a$. Then sublinear growth of $B,F$, {\em Step 1} and \eqref{eq:zz} imply
 that assumption (ii) is satisfied. By the basic energy estimate for parabolic equations \eqref{44} we obtain that $y\in L^\infty(0,T;L^2)\cap L^2(0,T;W^{1,2}_0)$ a.s. Besides, since $u\in C([0,T];L^2)$ a.s. due to our assumptions and the same is valid for $z$ due to \eqref{eq:zzz}, assumption (i) holds true as well. We may conclude
\begin{equation}\label{yinfty}
\begin{split}
\|y\|_{C^{\alpha/2,\alpha}}&\leq K_1\big(\|u_0\|_{C^{\alpha}(\overline{D})}\\
&\quad+\|B(u)+(A(u)-\text{I})\nabla z\|_{L^{2r_0}(D_T)}+\|F(u)\|_{L^{r_0}(D_T)}\big)
\end{split}
\end{equation}
for $\alpha\leq \iota$ where $\alpha$ is given by Theorem \ref{thm:lady01}.
Therefore
\begin{equation}\label{yholder}
 \begin{split}
\E  \|y\|^m_{C^{\alpha/2,\alpha}} &\leq K_1\big( \E\|u_0\|^{m}_{C^{\alpha}(\overline{D})} \\
&\hspace{-3mm}+\E\|B(u)+(A(u)-\text{I})\nabla z\|^{m}_{L^{2r_0}(D_T)}+\E\|F(u)\|^{m}_{L^{r_0}(D_T)}\big).
 \end{split}
\end{equation}
We now use \eqref{ladyzenbound1}$-$\eqref{ladyzenbound2} to deduce that the above right hand side is finite thanks to our assumptions on $u$.


{\em Step 3: H\"{o}lder regularity of $z$.} In order to complete the proof it is necessary to improve the regularity of $z$. Recall that for all $m\in[2,\infty)$, the solution $u$ to $(\ref{equation})$ belongs to $L^m(\Omega;L^m(0,T;L^m))$ and that $H(u)$ belongs to $L^m(\Omega;L^m(0,T;\gamma(K,L^m)))$. We now apply Proposition \ref{RegulStochastic} - $(i)$ and (\hyperref[compgamma]{$\text{H}_{0,m}$}) to obtain, since $H^{a,r}_0\subset H^{a,r}$, that for $m\in(2,\infty)$, $\delta\in(0,1-2/m)$ and $\gamma\in [0,1/2 -1/m-\delta/2)$, ${z}\in L^m(\Omega;C^{\gamma}([0,T];H^{\delta,m}))$ with
$$\E\|{z}\|^m_{C^{\gamma}([0,T];H^{\delta,m})}\leq C\,\Big{(}1+\E\|u\|^m_{L^m(0,T;L^m)}\Big{)}.$$
Note that we can choose $\delta$ and $\gamma$ to be independent of $m$. For instance, let us suppose in the sequel that $m\geq 3$; then $\delta=1/6$ and $\gamma=1/12$ satisfies the conditions above for any $m\geq 3$. Furthermore, from now on, we also suppose that $m \geq 7N:=m_0$. This implies that $m\geq 3$ and $\delta m>N$, so that the following Sobolev embedding holds true
$$ H^{\delta,m} \hookrightarrow C^{\lambda},\quad \lambda:=\delta-N/m_0.$$
We conclude that, for all $m\geq m_0$,
\begin{equation}\label{zholder}
\E\|{z}\|^m_{C^{\gamma}([0,T];C^{\lambda})}\leq C\,\Big{(}1+\E\|u\|^m_{L^m(0,T;L^m)}\Big{)}<\infty.
\end{equation}
Note that for $m\in[2,m_0)$, we can write with the H\"{o}lder inequality
\begin{equation}\label{zholderbis}
\E\|{z}\|^m_{C^{\gamma}([0,T];C^{\lambda})}\leq \Big{(}\E\|{z}\|^{m_0}_{C^{\gamma}([0,T];C^{\lambda})}\Big{)}^{m/m_0}<\infty.
\end{equation}

{\em Step 4: Conclusion.} Finally, we set $\eta:=\text{min}(\alpha,2\gamma,\lambda)>0$ and we recall that $u=y+{z}$ so that the conclusion follows from \eqref{yholder}, \eqref{zholder}, \eqref{zholderbis} due to the fact that $C^{\eta/2}([0,T];C^{\eta}(\overline D))\subset C^{\eta/2,\eta}([0,T]\times \overline D)$.
\end{proof}

\section{Increasing the regularity: proof of Theorem \ref{mainresult2}}
\label{sec:increase}

In this final section, we complete the proof of Theorem \ref{mainresult2}. Having Theorem \ref{mainresult1} in hand, it is now possible to significantly increase the regularity of $u$ using the results given in Subsection \ref{s3.2}. 

We treat differently the cases $k=1, \,k=2 $ and $k=3,4$.

\subsection{The case $k=1$}
\label{subsec:k1}


The proof is divided in two parts: we first increase the regularity in space and then in time.

{\em Step 1: Regularity of $z$.} First, we improve the regularity of $z$ that was defined in \eqref{defz}.
According to Theorem \ref{mainresult1}, there exists $\eta>0$ such that for all $m\in[2,\infty)$, $u\in L^m(\Omega;C^{\eta/2,\eta}(\overline{D_T}))$. In particular, since $u$ satisfies Dirichlet boundary conditions, this implies that $u\in L^m(\Omega;L^m(0,T;H^{\kappa,m}_0))$ provided $\kappa<\eta$. With (\hyperref[compgamma]{$\text{H}_{\kappa,m}$}), we deduce that $H(u)\in L^m(\Omega;L^m(0,T;\gamma(K,H^{\kappa,m}_0)))$. An application of Proposition \ref{RegulStochastic} yields that $z\in L^m(\Omega;C^{\gamma}([0,T];H_0^{\kappa+\delta,m}))$ for every $m\in(2,\infty)$ with
$$\E\|z\|^m_{C^{\gamma}([0,T];H_0^{\kappa+\delta,m})}\leq C\,\Big{(}1+\E\|u\|^m_{L^m(0,T;H^{\kappa,m}_0)}\Big{)},$$
where $\delta\in(0,1-2/m)$ and $\gamma\in[0,1/2-1/m-\delta/2)$. In the sequel, we assume that $m\geq (N+4)/\kappa:=m_0$. Then $\delta=1-3/m_0$ and $\gamma=1/(4m_0)$ satisfies the conditions above uniformly in $m\geq m_0$. Furthermore, observe that $(\kappa+\delta)m> \kappa m\geq \kappa m_0\geq N$ so that the following Sobolev embedding holds true
$$ H^{\kappa+\delta,m} \hookrightarrow C^{\sigma},\quad \sigma = \kappa+\delta-N/m_0.$$
Besides, by definition of $m_0$, $\sigma=\kappa+1-(N+3)/m_0>1$.

We deduce that there exists $\gamma>0,\, \sigma >1$ such that for all $m\geq m_0$, $z\in L^m(\Omega;C^{\gamma}([0,T];C^{\sigma}(\overline D)))$ with
\begin{equation}\label{eq:estimz}
\E\|z\|^m_{C^{\gamma}([0,T];C^{\sigma}(\overline D))}\leq C\,\Big{(}1+\E\|u\|^m_{L^m(0,T;H^{\kappa,m}_0)}\Big{)}.
\end{equation}

{\em Step 2: Regularity of $y$.} Next, we improve the regularity of $y$ that is given by \eqref{diff}. 

As a consequence of Theorem \ref{mainresult1}, \eqref{ladyzenbound2} and \eqref{eq:estimz}, we obtain due to the assumptions upon $A,\,B$ and $F$ that, for all $m\in[2,\infty)$
\begin{equation*}
\begin{split}
A(u)&\in L^m(\Omega\semicol C^{\alpha/2,\alpha}(\overline{D_T}),\\
B(u)+(A(u)-\mathrm{I})\nabla z&\in L^m(\Omega\semicol C^{\alpha/2,\alpha}(\overline{D_T}),\\
F(u)&\in L^m(\Omega\semicol L^m(0,T\semicol L^m)),\\
u_0&\in L^m(\Omega;C^{1+\alpha}(\overline{D})),
\end{split}
\end{equation*}
where $\alpha:=\text{min}(\iota,\eta,\sigma-1,2\gamma)>0$. Thus the hypotheses of Theorem 
\ref{thm:schauder} are fulfilled and we obtain the following (pathwise) estimate
\begin{equation*}
 \begin{split}
\|y\|_{C^{(1+\alpha)/2,1+\alpha}}&\leq K_2\, P_1(a_{\alpha/2,\alpha})\Big(\|u_0\|_{C^{1+\alpha}(\overline D)} +\|B(u)+(A(u)-\mathrm{I})\nabla z\|_{C^{\alpha/2,\alpha}} \\
&\qquad\qquad\quad +\|F(u)\|_{L^r(0,T\semicol L^r)}\Big),  
 \end{split}
\end{equation*}
where $r\in[2,\infty)$ is large enough. We conclude that, for all $m\in[2,\infty)$,
\begin{equation}\label{newy}
y\in L^m(\Omega\semicol C^{(1+\alpha)/2,1+\alpha}(\overline{D_T}))
\end{equation}
which together with \eqref{eq:estimz} yields $u\in L^m(\Omega\semicol C^{\gamma,1+\alpha}(\overline{D_T})).$

{\em Step 3: Time regularity.} Having in hand the improved regularity of $u$, we consider again the stochastic convolution $z$, repeat the approach from the first step of this proof and obtain due to Theorem \ref{mainresult1} (with $\delta=0$) and (\hyperref[compgamma]{$\text{H}_{1+\kappa,m}$})
\begin{equation}\label{newz}
\begin{split}
\E&\|z\|^m_{C^{\lambda}([0,T];H_0^{1+\kappa,m})}\\
&\leq C\,\Big{(}1+\E\|u\|^m_{L^m(0,T;H^{1+\kappa,m}_0)}+\stred\|u\|^{(1+\kappa)m}_{L^{(1+\kappa)m}(0,T\semicol H^{1,(1+\kappa)m}_0)}\Big{)}<\infty,
\end{split}
\end{equation}
where $\kappa<\alpha$ and $\lambda\in(0,1/2-1/m)$. Therefore for any $\lambda\in(0,1/2)$ there exists $m_0$ large enough so that \eqref{newz} holds true for any $m\geq m_0$ and the Sobolev embedding then implies that $z\in L^m(\Omega\semicol C^\lambda([0,T]\semicol C^{1+\beta}(\overline D)))$ for $\beta<\kappa$. Since we already have \eqref{newy} we deduce that $u\in L^m(\Omega\semicol C^{\lambda,1+\alpha}(\overline{D_T}))$ for any $\lambda\in(0,1/2)$ and $m\in \N$. 

{\em Step 4: Conclusion.} It is now possible to reproduce the 3 steps above. In step 1, we can now 
take $\kappa < 1+\alpha$. Then 
$$
\E\|z\|^m_{C^{\gamma}([0,T];C^{\sigma})}\leq C\,\Big{(}1+\E\|u\|^m_{L^m(0,T;H^{\kappa,m}_0)}\Big{)}.
$$
with $\sigma = \kappa+\delta -N/m$, $\gamma<1/2-1/m-\delta/2$, $\delta <1-2/m$.

Let $\eps<(1-\iota)/2$, $ m\ge \max( 4,2N) 1/\eps$ and $\delta = (1-\alpha)/2$, then we can 
take $\sigma-1=2\gamma =(1-\alpha)/2-\eps$. Thus in step 2, we can reproduce the argument 
with $\alpha$ replaced by $\alpha_1=\min(\iota, (1-\alpha)/2-\eps)$ and conclude $u\in L^m(\Omega\semicol C^{\lambda,1+\alpha_1}(\overline{D_T}))$ for any $\lambda\in(0,1/2)$ and $m\in \N$. 
If $\alpha_1<\iota$, we reproduce this argument and define recursively $\alpha_{n+1}=\min(\iota, (1-\alpha_n)/2-\eps)$. In a finite number of step, we have $\alpha_n=\iota$.

The proof is complete for $k=1$.
\begin{remark}\label{r5.1}
Note that reproducing step 1, we can  finally prove that $z\in L^m(\Omega;C^{\gamma}([0,T];H_0^{\kappa+\delta,m}))$
for  $\kappa<1+\iota$, $\gamma<1/2-1/m-\delta/2$, $\delta <1-2/m$ and any $m\in \N$. In particular,
it is possible to take $\kappa+\delta-1/m>2$ and we deduce that $z$ and its first and second derivatives vanish on $\partial D$.
\end{remark}

\subsection{The case $k=2$}

{\em Step 1: Regularity of $z$.} Again we first increase the regularity of $z$. We know that for any $\lambda\in (0,1/2)$, there exists $\beta>0$ such that for all $m\in[2,\infty)$, $u\in L^m(\Omega;C^{\lambda,1+\beta}(\overline{D_T}))$. We deduce
\begin{equation*}
\begin{split}
\E\|z&\|^m_{C^{\gamma}([0,T];H^{1+\kappa+\delta,m})}\\
&\leq C\,\Big{(}1+\E\|u\|^m_{L^m(0,T;H^{1+\kappa,m})}+\stred\|u\|^{(1+\kappa)m}_{L^{(1+\kappa)m}(0,T\semicol H^{1,(1+\kappa)m})}\Big{)}<\infty,
\end{split}
\end{equation*}
where $\kappa<\beta$, $\delta\in(0,1-2/m)$ and $\gamma\in[0,1/2-1/m-\delta/2)$. By a similar reasoning as above we obtain due to the Sobolev embedding that there exist $\gamma>0$ and
$\sigma>2$ such that  $z\in L^m(\Omega\semicol C^{\gamma}([0,T];C^\sigma(\overline D)))$ 
for $m\in[2,\infty)$.

{\em Step 2: Regularity of $y$.} In order to improve the space regularity of $y$ we use Theorem
\ref{t3.3}. In particular, we set
$$a_{ij}=A_{ij}(u), , \quad f=\sum_{ij}A'_{ij}(u)\partial_iu\partial_j y + \mathrm{div}\big(B(u)+(A(u)-\mathrm{I})\nabla z\big)+F(u).$$
According to the results above, we have
\begin{equation}\label{scalingfail}
\begin{split}
a_{ij},\, f&\in L^m(\Omega\semicol C^{\alpha/2,\alpha}(\overline{D_T})),\\
u_0 &\in L^m(\Omega\semicol C^{2+\alpha}(\overline{D})),
\end{split}
\end{equation}
for some $\alpha \in (0,\sigma-2]$ and all $m\in[2,\infty)$ provided $A,\,B\in C^2_b,\, F\in C^1_b$. The compatibility conditions (iv) are: $u_0=0$ and $-\sum_{ij}A_{ij}(u)\partial_{ij}u
=\sum_{ij}A'_{ij}(u)\partial_iu\partial_j y + \mathrm{div}\big(B(u)+(A(u)-\mathrm{I})\nabla z\big)+F(u)=0$,
on $\{0\}\times \partial D$. The first one is clearly satisfied. For the second one, we use Remark
\ref{r5.1} to rewrite it as:
$-\sum_{ij}A_{ij}(0)\partial_{ij}u_0
=\sum_{ij}A'_{ij}(0)\partial_iu_0\partial_j u_0 + \mathrm{div}\big(B(u_0)\big)+F(u_0)=0$, rearranging the terms this is exactly $u_0^{(1)}$. 
Thus,  Theorem \ref{t3.3} applies and we deduce
$$y\in L^m(\Omega\semicol C^{\alpha/2+1,\alpha+2}(\overline{D_T})),$$
hence
$$u\in L^m(\Omega\semicol C^{\gamma,\alpha+2}(\overline{D_T})).$$

{\em Step 3: Time regularity.} Finally, we improve the time regularity of $u$ by considering the stochastic convolution again as in Subsection \ref{subsec:k1}. We obtain that for any $\lambda\in(0,1/2)$ there exists $m_0$ large enough so that
\begin{equation*}
\begin{split}
\E\|z&\|^m_{C^{\lambda}([0,T];H^{2+\kappa,m})}\\
&\leq C\,\Big{(}1+\E\|u\|^m_{L^m(0,T;H^{2+\kappa,m}_0)}+\E\|u\|^{(2+\kappa)m}_{L^{(2+\kappa)m}(0,T;H^{1,(2+\kappa)m}_0)}\Big{)},
\end{split}
\end{equation*}
holds true for any $m\geq m_0$ and the Sobolev embedding then implies that $z\in L^m(\Omega\semicol C^\lambda([0,T]\semicol C^{2+\beta}(\overline D)))$ for $\beta<\kappa$. 

Finally, we iterate the argument as in step 4 of the case $k=1$ and 
this completes the proof.

\subsection{The case $k=3,4$}

The case $k=3$ is treated exactly as above using Theorem \ref{t3.3} except that \eqref{scalingfail} is replaced by
\begin{equation}
\begin{split}
a_{ij},\,f&\in L^m(\Omega\semicol C^{\alpha/2,\alpha+1}(\overline{D_T}))
=L^m(\Omega\semicol C^{(\alpha+1)/2,\alpha+1}(\overline{D_T})),\\
u_0 &\in L^m(\Omega\semicol C^{\alpha+3}(\overline{D})),
\end{split}
\end{equation}
for some $\alpha\in (0,\sigma-3]$ where $\sigma > 3$.

For $k=4$, we argue similarly but apply Theorem \ref{t3.4}. The only thing we need to check is that 
$f$ is smooth in time on the boundary and the
validity of the third compatibility condition in (iii). 
Recall that $f=\sum_{ij}A'_{ij}(u)\partial_iu\partial_j y + \mathrm{div}\big(B(u)+(A(u)-\mathrm{I})\nabla z\big)+F(u)$. Since $z$ and its derivatives up to order $4$ vanish on $\partial D$, we have
$$
f|_{S_T}=\sum_{ij}A'_{ij}(y)\partial_iy\partial_j y + \mathrm{div}\big(B(y)\big)+F(y).
$$
In the proof for $k=3$, we get in Step 2 that $y\in L^m(\Omega\semicol C^{1+\iota/2,3+\iota}(\overline{D_T}))$, this 
shows that $f|_{S_T}\in L^m(\Omega\semicol C^{1+\iota/2,2+\iota}(\overline{S_T}))$. Concerning the compatibility condition, we take $\cA=\sum_{ij}A_{ij}(u)\partial_{ij}$. Thus, on the boundary 
$\cA=\sum_{ij}A_{ij}(0)\partial_{ij}$. In particular, it is constant in time and the last term of the 
compatibility condition vanishes. Moreover on $\{0\}\times \partial D$, $\partial_t u =u_0^{(1)}=0$, we deduce:
$$
\partial_t f= 2\sum_{ij}A'_{ij}(0)\partial_iu^{(1)}_0\partial_j u_0 + B'(0)\cdot \nabla u_0^{(1)}.
$$
Also, still on $\{0\}\times \partial D$,
$$
\cA f=\sum_{lk}A_{lk}(0)\partial_{lk}\left( \sum_{ij}A'_{ij}(u_0)\partial_iu_0\partial_j u_0 + \mathrm{div}\big(B(u_0)\big)+F(u_0)\right)
$$
and, since $\cA=\cA_0$,
$$
\cA^2u_0+\cA f= \cA_0\left(  \sum_{ij}A_{ij}(u_0) \partial_{ij}u_0+A'_{ij}(u_0)\partial_iu_0\partial_j u_0 + \mathrm{div}\big(B(u_0)\big)+F(u_0)\right).
$$
Regrouping terms, we obtain:
$$
\cA^2u_0+\cA f= \cA_0u_0^{(1)}.
$$
Therefore our assumption in Theorem \ref{mainresult2} implies the compatibility condition.


\end{document}